\documentclass{amsart}
\usepackage{amsfonts,amssymb,verbatim,amsmath,amsthm,latexsym,textcomp,amscd,hyperref}
\usepackage{latexsym,amsfonts,amssymb,epsfig,verbatim}
\usepackage{amsmath,amsthm,amssymb,latexsym,graphics,textcomp}
\usepackage{graphicx}
\usepackage{color}
\usepackage{url}
\usepackage{pgfplots}
\usepackage{tikz}
\usepackage{tikz-cd}

\begin{document}

\newtheorem{theorem}{Theorem}[section]
\newtheorem{prop}[theorem]{Proposition}
\newtheorem{lemma}[theorem]{Lemma}
\newtheorem{cor}[theorem]{Corollary}
\newtheorem{prob}[theorem]{Problem}
\newtheorem{defn}[theorem]{Definition}
\newtheorem{notation}[theorem]{Notation}
\newtheorem{fact}[theorem]{Fact}
\newtheorem{conj}[theorem]{Conjecture}
\newtheorem{claim}[theorem]{Claim}
\newtheorem{example}[theorem]{Example}
\newtheorem{rem}[theorem]{Remark}
\newtheorem{assumption}[theorem]{Assumption}
\newtheorem{scholium}[theorem]{Scholium}
\newtheorem{qn}[theorem]{Question}
\newtheorem{conv}[theorem]{Notation and Convention}

\newcommand{\HAT}{\widehat}
\newcommand{\map}{\rightarrow}
\newcommand{\C}{\mathcal C}
\newcommand\AAA{{\mathcal A}}
\def\AA{\mathcal A}

\def\L{{\mathcal L}}
\def\al{\alpha}
\def\A{{\mathcal A}}

\newcommand\GB{{\mathbb G}}
\newcommand\BB{{\mathcal B}}
\newcommand\DD{{\mathcal D}}
\newcommand\EE{{\mathcal E}}
\newcommand\FF{{\mathcal F}}
\newcommand\GG{{\mathcal G}}
\newcommand\HH{{\mathbb H}}
\newcommand\II{{\mathcal I}}
\newcommand\JJ{{\mathcal J}}
\newcommand\KK{{\mathcal K}}
\newcommand\LL{{\mathcal L}}
\newcommand\MM{{\mathcal M}}
\newcommand\NN{{\mathbb N}}
\newcommand\OO{{\mathcal O}}
\newcommand\PP{{\mathcal P}}
\newcommand\QQ{{\mathbb Q}}
\newcommand\RR{{\mathbb R}}
\newcommand\SSS{{\mathcal S}}
\newcommand\TT{{\mathcal T}}
\newcommand\UU{{\mathcal U}}
\newcommand\VV{{\mathcal V}}
\newcommand\WW{{\mathcal W}}
\newcommand\XX{{\mathcal X}}
\newcommand\YY{{\mathcal Y}}
\newcommand\YB{{\mathbb Y}}
\newcommand\ZZ{{\mathcal Z}}
\newcommand\ZI{{\mathbb Z}}
\newcommand\hhat{\widehat}
\newcommand\flaring{{Corollary \ref{cor:super-weak flaring} }}
\newcommand\pb{\bar{p}_B}
\newcommand\pp{\bar{p}_{B_1}}
\newcommand{\eg}{\overline{EG}}
\newcommand{\eh}{\overline{EH}}
\def\Ga{\Gamma}
\def\Z{\mathbb Z}

\def\diam{\operatorname{diam}}
\def\dist{\operatorname{dist}}
\def\hull{\operatorname{Hull}}
\def\id{\operatorname{id}}
\def\Im{\operatorname{Im}}

\def\barycenter{\operatorname{center}}

\def\length{\operatorname{length}}
\newcommand\RED{\textcolor{red}}
\newcommand\BLUE{\textcolor{blue}}
\newcommand\GREEN{\textcolor{green}}
\def\mini{\scriptsize}

\def\acts{\curvearrowright}
\def\embed{\hookrightarrow}

\def\ga{\gamma}
\newcommand\la{\lambda}
\newcommand\eps{\epsilon}
\def\geo{\partial_{\infty}}
\def\bhb{\bigskip\hrule\bigskip}

\title[Coned-off spaces and CT maps]{On geometry of coned-off spaces and Cannon-Thurston maps}

\author{Pranab Sardar}
\address{Department of Mathematical Sciences,
	Indian Institute of Science Education and Research Mohali,
	Knowledge City, Sector 81, SAS Nagar,
	Punjab 140306,  India}

\email{psardar@iisermohali.ac.in}

\author{Ravi Tomar}
\email{ravitomar547@gmail.com}
\address{Department of Mathematical Sciences,
	Indian Institute of Science Education and Research Mohali,
	Knowledge City, Sector 81, SAS Nagar,
	Punjab 140306,  India}


\subjclass[2010]{20F65, 20F67 (Primary), 30F60(Secondary) }

\keywords{Complexes of groups, Cannon-Thurston map, hyperbolic groups, acylindrical action}

\date{\today}

\begin{abstract} 
A typical question addressed in this paper is the following. Suppose $Z\subset Y\subset X$ are
hyperbolic spaces where $Z$ is quasiconvex in both $Y$ and $X$. Let $\HAT{Y}$ and $\HAT{X}$
denote the spaces obtained from $Y$ and $X$ respectively by coning off $Z$ as defined
by Farb (\cite{farb-relhyp}). {\em If the inclusion of the coned-off spaces $\HAT{Y}\map \HAT{X}$ 
admits the Cannon-Thurston (CT) map then does the inclusion $Y\map X$ also admit the Cannon-Thurston map?}
The main result of this paper answers this question affirmatively provided 
$\HAT{Y}\map \HAT{X}$ satisfies Mitra's criterion (see Lemma \ref{mitra's criterion}) 
for the existence of CT maps, although the answer in general is negative.

The main application of our theorem is in the context of acylindrical complexes of hyperbolic groups. 
In \cite{martin1} A. Martin proved a combination theorem for developable, acylindrical complexes of 
hyperbolic groups. Suppose $(\mathcal G, \YY)$ is an acylindrical complex of hyperbolic groups with 
universal cover $B$ which satisfy the hypotheses of Martin's theorem. Suppose $\YY_1\subset \YY$ is a 
connected subcomplex such that the subcomplex of groups $(\mathcal G, \YY_1)$ also satisfies the hypotheses 
of Martin's theorem, it has universal cover $B_1$ and the natural homomorphism 
$\pi_1(\mathcal G, \YY_1)\map \pi_1(\mathcal G, \YY)$ is injective. It follows from the main theorem of 
this paper that the inclusion $\pi_1(\mathcal G, \YY_1)\map \pi_1(\mathcal G, \YY)$ admits 
the CT map if the inclusion $B_1\rightarrow B$ satisfies Mitra's criterion. Also
$\pi_1(\mathcal G, \YY_1)$ is quasiconvex in $\pi_1(\mathcal G, \YY)$ if in addition $B_1$ is qi embedded in $B$.
\end{abstract}
\maketitle


\section{Introduction} 
Suppose $X$ is a $\delta$-hyperbolic metric space and $\{A_i\}$ is a collection of $k$-quasiconvex subsets of $X$.
In \cite{mj-dahmani} Dahmani and Mj proved the following result.

\smallskip
{\em {\bf Proposition.} (1) \cite[Proposition 2.10]{mj-dahmani}
The coned-off space $\HAT{X}$ obtained from $X$ by coning off the sets $A_i$'s is hyperbolic.

(2) \cite[Proposition 2.11]{mj-dahmani} If $A\subset X$ is quasiconvex then $A$ is quasiconvex in $\HAT{X}$ as well.
}

\smallskip
They also proved a partial converse (see \cite[Proposition 2.12]{mj-dahmani}) to (2).
However, these results motivated us to ask the following.

\smallskip
\begin{qn}\label{intro qn 1}
Suppose $X$ is a hyperbolic metric space and $\{B_i\}$ is a collection of uniformly quasiconvex 
subsets of $X$ and $Y\subset X$ which is hyperbolic and properly embedded in $X$ with the induced 
length metric from $X$. 
Suppose $\{A_j\}$ is a collection of subsets of $Y$ which are uniformly quasiconvex in both $X$ 
and $Y$. Suppose each $A_j$ is contained in $B_i\cap Y$ for some $B_i$. Then by coning off the various
$B_i$'s and $A_j$'s we have two hyperbolic spaces $\HAT{Y}, \HAT{X}$ by the above proposition. 
Let $\iota:Y\map X$ and $\HAT{\iota}:\HAT{Y}\map \HAT{X}$ denote the inclusion maps.

(1) Suppose the inclusion $\HAT{\iota}: \HAT{Y}\map \HAT{X}$ admits the Cannon-Thurston (CT) map. Does the 
inclusion $\iota: Y\map X$ also admit the CT map?

(2) Suppose the inclusion $\iota: Y\map X$ admits the CT map. Does the inclusion $\HAT{\iota}: \HAT{Y}\map \HAT{X}$ 
admit the CT map?
\end{qn}

We recall that the notion of Cannon-Thurston maps or CT maps was
inspired from the work of Cannon and Thurston (\cite{CTpub}) on hyperbolic $3$-manifolds. 
It was introduced in Geometric Group Theory by Mahan Mitra (Mj) in \cite{mitra-trees} and \cite{mitra-ct}. We shall use the
following (equivalent) variant of it in this paper:

{\em Given a map $f:Y\map X$ between two (Gromov) hyperbolic spaces, the CT map for $f$ is
a} continuous extension {\em of $f$ to the Gromov boundaries $\partial f: \partial X\map \partial Y$.}

Since in this context even $f$ need not be continuous, we have emphasized the phrase `continuous extension' which 
would require an explanation. One is refered to Section \ref{section: CT maps} for a somewhat detailed discussion. 
In particular, if $H<G$ are hyperbolic groups, one 
asks if the inclusion $H\map G$ admits the CT map $\partial H\map \partial G$. This question of Mahan Mitra (Mj) 
appeared in \cite[Question 1.19]{bestvinahp}
which was subsequently answered in the negative by Baker and Riley (\cite{baker-riley}).
The CT maps in Geometric Group Theory became popular mostly through the work of 
Mahan Mitra (Mj). One is referred to \cite{mahan-icm} for a detailed history of 
the development. However, in this paper we obtain partial answers to both 
the parts of Question \ref{intro qn 1}:

\smallskip
{\em {\bf Theorem 1.} \textup{(Theorem \ref{electric ct})}
Suppose we have the hypotheses of Question \ref{intro qn 1}. In addition suppose
$\{B_i\}$ is a locally finite collection of subsets of $X$.

(1) If the inclusion $\HAT{\iota}: \HAT{Y}\map \HAT{X}$ satisfies Mitra's criterion then the 
inclusion $\iota:Y\map X$ admits the CT map.

(2) Suppose  the CT map $\partial \iota: \partial Y\map \partial X$ exists. 
Then the CT map $\partial \HAT{\iota}:\partial \HAT{Y}\map \partial \HAT{X}$ exists if and only if for any $B_i$, and
$\xi\in \Lambda_X(B_i)\cap Im(\partial \iota)$ we have $\xi \in \Lambda_X(A_j)$ for some $A_j$.
 }

\smallskip

Mitra's criterion appears in \cite[Lemma 2.1]{mitra-trees} which gives a sufficient condition for the
existence of CT maps. See Lemma
\ref{mitra's criterion} of this paper for the statement. Also, 
we recall that if $Z$ is a metric space then a collection of subsets $\{Z_i\}$ of $Z$
is called {\em locally finite} for any $z\in Z$ and $R>0$, $B(z,R)$ intersects only
finitely many $Z_i$. However, we show by an example (see Example \ref{counter example}) that 
in the first part of the above theorem `Mitra's criterion' cannot be
replaced by mere existence of CT maps. 

On the other hand, in the course of the proof of Theorem $1$ we prove the following result
describing a relation between $\partial X$ and $\partial \HAT{X}$. We note that some parts of this
theorem were proved earlier by other authors. Comments on those works are included after the
statement of the theorem.

\smallskip
{\em{\bf Theorem 2.} 
Given $k\geq 0, \delta \geq 0$ there are $K\geq 0,R\geq 0$ such that the following hold:\\
Suppose $X$ is a $\delta$-hyperbolic metric space and $\{B_i\}$ is a collection of $k$-quasiconvex subsets of $X$.
Let $\HAT{X}$ denote the coned-off space obtained by coning the $B_i$'s. Let $\iota:X\map \HAT{X}$ denote the
inclusion map. Suppose $\gamma$ is a geodesic ray in $X$. Then we have the following:

(1) If $\gamma$ is an unbounded subset of $\HAT{X}$ then there is a $K$-quasigeodesic ray $\beta$ of
$\HAT{X}$ having the same starting point as that of $\gamma$ such that $Hd_{\HAT{X}}(\gamma, \beta)\leq R$. 
In particular, $\{\gamma(n)\}$ converges to a point of $\partial \HAT{X}$

{\em Let $\partial_h X$ denote the set of all points $\xi \in \partial X$ such that for
any geodesic ray $\gamma$ in $X$ with $\gamma(\infty)=\xi$, $\iota(\gamma)$ is an unbounded subset of $\HAT{X}$.}

(2) The map $\partial_h X\map \partial \HAT{X}$ described in (1) is a homeomorphism
where $\partial_h X$ has the subspace topology from $\partial X$.

(3) If in addition the collection $\{B_i\}$ is locally finite then for any geodesic $\gamma$ of $X$
which is bounded in $\HAT{X}$, $\gamma(\infty)\in \Lambda_X(B_j)$ for some $B_j$.

In particular in this case we have} $$\partial X= \partial_h X\bigcup (\cup_i \Lambda_X(B_i)).$$ 

Here $\Lambda_X(B_i)$ denotes the limit set of $B_i$ in $\partial X$ for any $B_i$.
However, part (1) of Theorem $2$
was proved earlier by Kapovich and Rafi (\cite[Corollary 2.4]{kapovich-rafi}) and (2) was
proved by Dowdall and Taylor (\cite[Theorem 3.2]{dowdall-taylor}). A statement similar to
part (3) was proved by Abbott and Manning (\cite[Lemma 6.12]{abbott-manning}). However, our proofs for
all the parts of Theorem $2$ are completely independent of the proofs given by the previous authors.
The main ingredient of our proofs is the second part of the above proposition of Dahmani and Mj.
In this paper, (1) of Theorem $2$ is proved as Lemma \ref{hor rays}, part (2) as 
Theorem \ref{electric cor}; first statement of part (3) is proved in Proposition \ref{electric main lemma},
and the last part is the content of Theorem \ref{elec union thm}.

Subsequently, the main set of examples of hyperbolic groups to which we apply Theorem $1$
comes from the following context.

\smallskip
{\bf Two problems on complexes of hyperbolic groups}\\
Motivated by the celebrated combination theorem of Bestvina and Feighn \cite{BF} for graphs of hyperbolic 
groups,  Misha Kapovich (\cite[Problem 90]{kapovich-problem}) asked if one can prove a combination theorem for complexes of hyperbolic groups.
We recall the relevant definitions and results about complexes of groups in Section \ref{4}.
However, one may pose an important special case of M. Kapovich's problem as follows. 

\smallskip
{\bf Problem 1.} {\em Let $(\mathcal{G},\YY)$ be a developable complex of groups such that
the following hold:

\noindent
(a) $\YY$ is a finite connected simplicial complex,\\
(b) local groups are hyperbolic and local maps are quasiisometric embeddings,\\
(c) the universal cover of $(\mathcal{G},\YY)$ is a hyperbolic space.

Then find sufficient conditions under which $\pi_1(\mathcal{G},\YY)$ is a hyperbolic group. }

\smallskip
\begin{rem}
We shall refer to a complex of groups $(\GG,\YY)$ satisfying (a), (b), (c) of Problem 1 to be a
{\bf complex of hyperbolic groups}. 
\end{rem}
For graphs of hyperbolic groups Bestvina-Feighn's hallway flaring condition proved to be necessary and sufficient, 
\cite{BF},\cite[Corollary 6.7]{gersten-necessary}.
Similar results follow from \cite{mahan-sardar} when the local maps are quasiisometries, i.e. isomorphisms
onto finite index subgroups. In contrast to that in \cite{martin1} (along with \cite[Corollary, p. 805]{martin-univ}) 
A. Martin proved a combination theorem for acylindrical complexes of hyperbolic groups.
We recall the statement of Martin's theorem in Section \ref{martin}.  

However, these combination theorems motivate the following.

\smallskip
{\bf Problem 2.} Suppose {\em $\YY_1\subset \YY$ are connected simiplicial complexes, $(\GG,\YY)$ is a
developable complex of groups and $(\GG,\YY_1)$ is the restriction of $(\GG,\YY)$ to $\YY_1$ such that 
the following hold:
\begin{enumerate}
\item $(\GG, \YY)$, $(\GG,\YY_1)$ are complexes of hyperbolic groups.
\item The groups $\pi_1(\GG,\YY)$ and $\pi_1(\GG,\YY_1)$ are both hyperbolic.
\item The induced homomorphism $\pi_1(\GG,\YY_1)\map \pi_1(\GG,\YY)$ is injective.
\item If $B$, $B_1$ are the universal covers of $(\GG,\YY)$ and $(\GG,\YY_1)$ 
respectively then the natural inclusion $B_1\map B$ admits the CT map.
\end{enumerate}
 Does the CT map exist for the inclusion $\pi_1(\GG,\YY_1)\map \pi_1(\GG,\YY)$ ?}

\smallskip
We note that in this case it is automatic that $(\GG,\YY_1)$ is developable 
\cite[Corollary 2.15, Chapter III.C]{bridson-haefliger} and it satisfies (a) and (b)
of Problem 1. Also (3) of Problem 2 implies that the natural map $B_1\map B$ is injective. However,
solution to this problem is unknown in general even where $\YY_1$ is a vertex of $\YY$. For a graph $\YY$ 
this was answered in the affirmative in \cite{mitra-trees}, in case $\YY_1$ is a vertex of $\YY$,  and
for any connected subgraph $\YY_1$ in \cite{kapovich-sardar}. When the local maps are all quasiisometries 
then this is known to be true for any (finite connected simplicial complex) $\YY$ 
due to the work of \cite{mahan-sardar} if $\YY_1$ is a vertex of $\YY$, and 
for any $\YY_1$ such that $B_1\map B$ is a quasiisometric embedding due to \cite{mbdl2}. 
There is no other case known to be true till date. 
However, we prove the following theorem as one of the applications of Theorem $1$.


\smallskip
\begin{comment}Here is the set up of the theorem we are going to state.
\begin{enumerate}
\item Suppose $(\mathcal{G},Y)$ is a complex of groups satisfying the hypothesis of Martin's theorem. 
\item Suppose $Y_1$ is a connected subcomplex of $Y$ and $(\mathcal{G},Y_1)$ is the subcomplex of 
groups obtained by restricting $(\mathcal{G},Y)$ to $Y_1$. 
\item Suppose $(\mathcal{G},Y_1)$ also satisfies the hypotheses of Martin's theorem.

Thus both $\pi_1(\mathcal{G},Y_1)=H$, say, and $\pi_1(\mathcal{G},Y)=G$, say, are hyperbolic.
\item Suppose the natural homomorphism $H=\pi_1(\mathcal{G},Y_1) \map G=\pi_1(\mathcal{G},Y)$ is injective, and

Let $B, B_1$ be the universal covers of $(\mathcal{G},Y)$  and $(\mathcal{G},Y_1)$ respectively.  
It then follows \RED{ref} that there is an embedding $B_1\map B$ of the universal covers.
\end{enumerate} \end{comment}

{\em {\bf Theorem 3.}\textup{(Theorem \ref{acyl ct thm})} \label{ct map}
Suppose $(\GG,\YY)$ is a developable complex of groups and $\YY_1\subset \YY$ is a connected subcomplex and
$(\GG,\YY_1)$ is the restriction of $(\GG,\YY)$ to $\YY_1$ such that the following hold:
\begin{enumerate}
\item $(\GG,\YY)$, $(\GG, \YY_1)$ are complexes of hyperbolic groups.
\item The groups $\pi_1(\GG,\YY)$ and $\pi_1(\GG,\YY_1)$ are both hyperbolic.
\item All the local groups of $(\GG,\YY)$ are quasiconvex in $\pi_1(\GG,\YY)$.
\item The induced homomorphism $\pi_1(\GG,\YY_1)\map \pi_1(\GG,\YY)$ is injective.
\item Suppose $B$, $B_1$ are the universal covers of $(\GG,\YY)$ and $(\GG,\YY_1)$ and
the natural inclusion $B_1\map B$ satisfies Mitra's criterion.

\end{enumerate}
Then the Cannon-Thurston map for the inclusion $\pi_1(\GG, \YY_1)\map \pi_1(\GG,\YY)$ exists.
Moreover, $\pi_1(\GG, \YY_1)$ is quasiconvex in $\pi_1(\GG,\YY)$  if and only if the Cannon-Thurston map for 
the inclusion $B_1\rightarrow B$ is injective.

In particular if (1)-(4) hold and $B_1$ is quasiisometrically embedded in $B$ then $H$ is quasiconvex in $G$. }

\smallskip
A particularly interesting instance where the hypotheses of Theorem $3$ are true is in the context of
acylindrical complexes of hyperbolic groups:

Suppose $(\GG,\YY)$ is a complex of groups and $\YY_1\subset \YY$ is a subcomplex 
and $(\GG,\YY_1)$ is the restriction of $(\GG,\YY)$ to $\YY_1$ such that the following hold:
\begin{enumerate}
\item $(\GG,\YY)$, $(\GG, \YY_1)$ are complexes of hyperbolic groups.
\item The universal cover $B$ and $B_1$ of $(\GG,\YY)$ and $(\GG,\YY_1)$ respectively are both CAT(0).
\item The action of $\pi_1(\GG,\YY)$ on $B$ is acylindrical.
\item The induced homomorphism $\pi_1(\GG,\YY_1)\map \pi_1(\GG,\YY)$ is injective.
\end{enumerate} 
Then it follows by Martin's theorem (see \cite[p. 34]{martin1} or Section \ref{martin} of this paper for the statement
of the theorem) 
that $\pi_1(\GG,\YY)$ is hyperbolic and the local groups
are quasiconvex in $\pi_1(\GG,\YY)$. Also clearly if the inclusion $B_1\map B$ is a proper 
embedding, for instance when the inclusion $B_1\map B$ satisfies Mitra's criterion, 
then the action of $\pi_1(\GG, \YY_1)$ on $B_1$ is acylindrical. In that case, again
by Martin's theorem $\pi_1(\GG, \YY_1)$ is hyperbolic. However, we have the following
theorem in this situation.

\smallskip
{\bf Theorem 4.} (Corollary \ref{martin application})
{\em Suppose the map $B_1\rightarrow B$ satisfies Mitra's criterion.

Then there exists a Cannon-Thurston map for the inclusion $\pi_1(\GG, \YY_1)\map \pi_1(\GG,\YY)$ and
$\pi_1(\GG,\YY_1)$ is quasiconvex in $\pi_1(\GG,\YY)$ if and only if the CT map for the inclusion 
$B_1\rightarrow B$ is injective. }

\smallskip

Lastly we prove the following special case of the above theorem where we can say more.

\smallskip

{\bf Theorem 5.} (Theorem \ref{edge qc}) {\em 
Suppose $\YY$ is a regular Euclidean polygon with at least four sides 
and $(\GG,\YY)$ is a complex of hyperbolic groups satisfying the conditions of Martin's theorem, i.e.
\begin{enumerate} 
\item $(\GG,\YY)$ is a complex of hyperbolic groups.
\item The universal cover of $(\GG,\YY)$ is CAT(0).
\item $\pi_1(\GG,\YY)$-action on the universal cover is acylindrical.
\end{enumerate}
If $\YY_1$ is an edge of $\YY$ then ($\pi_1(\GG,\YY)$ is hyperbolic and) the natural homomorphism
$\pi_1(\GG,\YY_1)\map \pi_1(\GG,\YY)$ is an injective qi embedding.}


\subsection{Outline of the paper.} 
Section \ref{2} reviews some definitions and results on hyperbolic spaces, in particular results
about quasiconvex subspaces, Gromov boundaries and CT maps. Section \ref{3} is the technical heart of the
paper where we discuss coned-off spaces and prove the main theorem of the paper. 
In Section \ref{section 4} we recall basic facts about complexes of groups and prove two theorems
about the existence of CT maps in the context of complexes of hyperbolic groups. 
Finally in Section \ref{5} we prove an interesting quasiconvexity theorem (Theorem \ref{edge qc}) and mention
a few examples to end the paper.\\

{\bf Acknowledgements:} The first author was partially supported by DST MATRICS grant (MTR/2017/000485) of the Govt of India. 
The second author was supported by the UGC research fellowship (Ref. No. 20/12/2015(ii)EU-V) of the Govt of India.

\section{Preliminaries}\label{2}
\subsection{Basic coarse geometric notions}
In this subsection we recall some standard notions of coarse geometry and set up
some notation and convention. 
Let $X$ be a metric space. For all $x,y\in X$ their distance in $X$ is denoted by $d_X(x,y)$ or
simply $d(x,y)$ when $X$ is understood. 
For any $A\subset X$ and $D\geq 0$ we denote the closed $D$-neighborhood, i.e. 
$\{x\in X: d(x,a)\leq D \mbox{ for some }\, a\in A\}$ by $N_D(A)$. For $A, B\subset X$ we shall denote by $d_X(A,B)$ the
quantity $\inf\{d_X(a,b):a\in A, b\in B\}$. The {\em Hausdorff distance} between $A,B$ in $X$ is defined to be
$Hd(A,B):= \inf\{D\geq 0: A\subset N_D(B), B\subset N_D(A)\}$. If $A\subset X$ we say that $A$ is {\em rectifiably connected}
in $X$ if for all $x,y\in A$, there is a path $\alpha:[0,1]\map X$ joining $x,y$ which is of finite length such that
$\alpha([0,1]) \subset A$. For a rectifiably connected subset $A\subset X$, by induced length metric on $A$ we mean
the length metric associated to the restriction of $d_X$ on $A$. We shall assume that this is a geodesic
metric as defined next.

Suppose $x,y \in X$. A \emph{geodesic (segment)} joining $x$ to $y$ is an 
isometric embedding $\alpha: [a,b] \map X$ where $[a,b]\subset \RR$ is an interval such that $\alpha(a)=x,\alpha(b)=y$. 
Most of the time we are interested only in the image of this embedding rather than the embedding itself.
We shall denote by $[x,y]_X$  or simply by $[x,y]$ the image of a geodesic joining $x,y$. 
If any two points of $X$ can be joined by a geodesic segment
then $X$ is said to be a {\em geodesic metric space} and $d$ is called a geodesic metric on $X$. 
In this paper, graphs are assumed to be connected
and it is assumed that each edge is assigned a unit length so that the graphs are naturally geodesic metric spaces 
(see \cite[Section 1.9, I.1]{bridson-haefliger}). Below are some other definitions important for us.

\begin{enumerate}
\item
Suppose that $X$ and $Y$ are two metric spaces and $\rho:[0,\infty)\map [0,\infty)$ is any map. A map $f:X\rightarrow Y$ is said to be a
$\rho$-\emph{proper embedding} if $d_Y(f(x),f(x'))\leq M$ implies $d_X(x,x')\leq \rho(M)$ for all $x,x'\in X$. A map $f:X\map Y$
is called a proper embedding if it is a $\rho$-proper embedding for some map $\rho:[0,\infty)\map [0,\infty)$.

In all instances in this paper the space $X$ is a subspace of $Y$, $f$ is the inclusion map and the metric on $X$ is the
induced length metric \cite[Definition 3.3, I.3]{bridson-haefliger} from $Y$. {\em Throughout the paper we have used $\iota$
to denote inclusion maps.}

\item If $L\geq 0$ then an {\em $L$-Lipschitz}
map $f:X\map Y$ between two metric spaces is one such that $d_Y(f(x),f(x'))\leq Ld_X(x,x')$ for all
$x,x'\in X$. A $1$-Lipschitz map will simply be called a Lipschitz map. 

A very common instance of  Lipschitz maps that will be encountered later is the following: If $Y$ is a length space
and $X$ is a subspace with the induced length metric from $Y$, then the inclusion map $\iota:X\map Y$ is Lipschitz. 

\item A map $f:X\map Y$ is said to be $L$-{\em coarsely Lipschitz} for a constant 
$L\geq 0$ if $d_Y(f(x), f(x'))\leq L+Ld_X(x,x')$ for all $x,x'\in X$. A map $f:X\map Y$ is said to be
coarsely Lipschitz if it is $L$-coarsely Lipschitz for some $L\geq 0$.

\item
Given $\lambda\geq 1,\epsilon\geq 0$, a map 
$f:X\rightarrow Y$ is said to be a \emph{$(\lambda,\epsilon)$-quasiisometric embedding} if for all $x,x'\in X$ we have,
$$\dfrac{1}{\lambda}d_X(x,x')-\epsilon\leq d_Y(f(x),f(x'))\leq \lambda d_X(x,x')+\epsilon.$$
A $(\lambda,\lambda)$-qi embedding is simply called a $\lambda$-qi embedding.

The map $f$ is said to be $(\lambda,\epsilon)$-\emph{quasiisometry} if $f$ is a \emph{$(\lambda,\epsilon)$-quasiisometric embedding} 
and moreover, $N_D(f(X))=Y$ for some $D\geq 0$.  

\item
 A $(\lambda,\epsilon)$-\emph{quasigeodesic} in a metric space $X$ is a \emph{$(\lambda,\epsilon)$-quasiisometric embedding} from an interval $I\subset \mathbb{R}$ in $X$. 

A (quasi)geodesic $\alpha:I\map X$ is called a {\em (quasi)geodesic ray} if $I=[0,\infty)$ and it is called a {\em (quasi)geodesic} 
line if $I=\RR$. 
A $(\lambda, \lambda)$-quasigeodesic segment or ray or line will simply be called a $\lambda$-quasigeodesic segment or
ray or line resepectively.
\end{enumerate}


{\bf Convention.} Occasionally we use phrases like  {\em uniform qi embedding}, {\em uniform quasigeodesic} etc
if (1) we do not need an explicit value of the corresponding parameters and (2) it is clear that there is 
such a value given the hypotheses of the lemma or proposition; e.g. see the second part of the following lemma.

\smallskip
The first part of the following lemma is very standard and the second part follows from the first immediately,
and hence we skip both their proofs. 
\begin{lemma}\label{quasigeodesic}
Given $L\geq 1, k\geq 1, \epsilon\geq 0$ and $\rho:[0,\infty)\map [0,\infty)$ we have constants
$K_{\ref{quasigeodesic}}=K_{\ref{quasigeodesic}}(L,\rho)$ and 
$C_{\ref{quasigeodesic}}=C_{\ref{quasigeodesic}}(\rho,k,\epsilon)$
such that the following hold:

(1) Suppose $X$, $Y$ are two metric spaces, and $f:Y\map X$ is an $L$-coarsely Lipschitz, $\rho$-proper embedding.
Suppose $Z$ is a geodesic metric space and $g:Z\map Y$ is such that $f\circ g$ is a $k$-qi embedding.
Then $g$ is a $K_{\ref{quasigeodesic}}$-qi embedding.

(2) Suppose $Y$ is a subspace of a metric space $X$ equipped with the induced 
length metric from $X$. If the inclusion $Y\map X$ is a $\rho$-proper embedding, 
then any $(k,\epsilon)$-quasigeodesic of $X$ contained in $Y$ is a
$C_{\ref{quasigeodesic}}$-quasigeodesic in $Y$.
\end{lemma}

The lemma below follows from an easy calculation, hence we skip its proof. 
\begin{lemma}\label{concat geod}
Given $D\geq 0$ there is $K_{\ref{concat geod}}=K_{\ref{concat geod}}(D)\geq 1$ such that the following holds:\\
Suppose $X$ is a geodesic metric space and $x_0\in X$. Suppose that $\{x_n\}$ is a sequence of points of $X$
such that for all $n\in \NN$ there is a geodesic $\alpha_n$ joining $x,x_n$ with the following properties:

(1) $x_i\in N_D(\alpha_n)$ for $1\leq i\leq n$;

(2) if $y_i$ is any nearest point of $\alpha_n$ from $x_i$
then $d(x,y_{i+1})>d(x,y_i)$ for $1\leq i\leq n-1$.

If $\beta_i$ is any geodesic in $X$ joining $x_i$ to $x_{i+1}$, for all $i\geq 0$ then the concatenation
of $\beta_i$'s is a $K$-quasigeodesic in $X$.
\end{lemma}

The following lemma is a basic exercise in point set topology and hence we skip its proof too.
\begin{lemma}\label{subseq}
Suppose $Z$ is a Hausdorff topological space, $z\in Z$ and $\{z_n\}$ is a sequence in $Z$. 
If for any subsequence $\{z_{n_k}\}_k$ of $\{z_n\}$, there exists a further subsequence $\{z_{n_{k_{l}}}\}_l$ 
of $\{z_{n_k}\}_k$ such that $\{z_{n_{k_{l}}}\}$ converges to $z$ then $\lim z_n= z$.
\end{lemma}
\subsection{Gromov hyperbolic spaces} 
In this subsection we briefly recall some definitions and results about (Gromov) hyperbolic metric spaces that will be
relevant for us. We refer the reader to Gromov's original article \cite{gromov-hypgps} as well as some of the standard references
like \cite{GhH}, \cite{bridson-haefliger} for more details. 

\begin{defn}
(1) By a geodesic polygon with $n$ sides in a geodesic metric space $X$ we mean a choice of 
$n$ points in $X$, say $x_1,x_2,\cdots, x_n$ and $n$ geodesic segments $[x_i,x_{i+1}]$, $1\leq i\leq n$
where $x_{n+1}=x_1$.

(2) Given $\delta\geq 0$, a geodesic triangle $\triangle$ in a geodesic metric space is said to be $\delta$-\emph{slim} if any side of $\triangle$ is contained in the union of $\delta$-neighborhood of remaining two sides. 

A geodesic metric space $X$ is said to be \emph{hyperbolic} if there exists $\delta\geq 0$ such that every geodesic 
triangle in $X$ is $\delta$-slim.

In this case the space is called $\delta$-hyperbolic.
\end{defn}

The following lemma is immediate from the definition of hyperbolic spaces.
\begin{lemma}\label{slim polygon}
A geodesic polygon with $n$ sides in a $\delta$-hyperbolic space is $(n-2)$-slim, i.e. 
any side of the polygon is contained in the $(n-2)\delta$-neighborhood of the union
of the remaining sides.
\end{lemma}

The conclusion of the following theorem is one of the most important properties of hyperbolic metric spaces.

\begin{theorem}{\em (\cite[Theorem 1.7, III.H]{bridson-haefliger}\textup{({\bf Stability of quasigeodesics})})}\label{stability} For $\delta\geq 0,\lambda\geq 1,\epsilon \geq 0$ there exists a constant 
$D_{\ref{stability}}=D_{\ref{stability}}(\delta,\lambda,\epsilon)$ with the following property:
	
	Let $X$ be a $\delta$-hyperbolic metric space. Then, the Hausdorff distance between a 
$(\lambda,\epsilon)$-quasigeodesic and a geodesic joining the same pair of end points is 
less than or equal to $D_{\ref{stability}}$. 	
\end{theorem}
It easily follows from this theorem that hyperbolicity is invariant under quasiisometry.

\begin{defn}
A finitely generated group $G$ is said to be {\em hyperbolic} if the Cayley graph of $G$ with respect to some (any) finite generating 
set is hyperbolic.
\end{defn}

It is a standard consequence of Milnor-{\u S}varc lemma \cite[Proposition 8.19, I.8]{bridson-haefliger} 
that the Cayley graphs 
of any group $G$ with respect to any two finite generating sets are quasiisometric. That hyperbolicity of a group
is well-defined follows from this and the fact that hyperbolicity is a qi invariant property
for geodesic metric spaces.


\begin{defn}[{\bf Quasiconvex subspaces}] \label{qc defn}Let $X$ be a (hyperbolic) metric space and let $A\subset X$. 
Let $K\geq 0$. Then $A\subset X$ is said to be $K$-{\em quasiconvex} in $X$ if any geodesic in $X$ with end points in $A$ is 
contained in $N_K(A)$. A subset $A$ of $X$ is said to be {\em quasiconvex} if it is $K$-quasiconvex for some $K\geq 0$. A $0$-quasiconvex subset is said to be a {\em convex} subset.

If $G$ is a hyperbolic group and $H<G$ then we say that $H$ is quasiconvex in $G$ if $H$ is a quasiconvex subset of a Cayley graph
of $G$.

\end{defn}
The following lemma gives natural examples of quasiconvex subspaces in hyperbolic spaces.
The proof of the first part of the lemma is immediate from the definition of quasiconvexity,
and the second part follows from stability of quasigeodesics. Hence we omit the proofs.

\begin{lemma}\label{qc lemma 1}
(1) Any geodesic in a $\delta$-hyperbolic space is $\delta$-quasiconvex.

(2) Any $(k,\epsilon)$-quasigeodesic in a $\delta$-hyperbolic space is 
$D_{\ref{stability}}(\delta,k,\epsilon)$-quasiconvex.
\end{lemma}

The next lemma shows persistences of quasiconvexity under qi embeddings of hyperbolic spaces. 
Thus it follows that quasiconvexity of any subset, for instance any subgroup, of a hyperbolic group 
is well-defined, i.e. independent of the Cayley graphs. 

\begin{lemma}\label{qi map qc}
Given $\delta\geq 0, k\geq 1, K\geq 0$ there is a constant 
$K_{\ref{qi map qc}}=K_{\ref{qi map qc}}(k,\delta, K)$ such that the following holds:\\
Suppose $f:X\map Y$ is a $k$-qi embedding of $\delta$-hyperbolic metric spaces.
If $A\subset X$ is $K$-quasiconvex then $f(A)\subset Y$ is $K_{\ref{qi map qc}}$-quasiconvex.
In particular, $f(X)$ is uniformly quasiconvex in $Y$.
\end{lemma}
\begin{proof}
	Let $\gamma$ be a geodesic in $X$ joining $x_1,x_2\in A$. Since $f$ is a $k$-qi embedding, 
$f( \gamma)$ is a $k$-quasigeodesic joining $y_i=f(x_i)$ for $i=1,2$. 
Hence, for any geodesic segment $\alpha$ joining 
$y_1, y_2$ in $Y$ we have $Hd(\alpha, f(\gamma))\leq D_{\ref{stability}}(\delta,k,k)$.
On the other hand, since $A$ is $K$-quasiconvex, $\gamma\subset N_K(A)$. Thus 
$f(\gamma)\subset N_{kK+k}(f(A))$ since $f$ is a $k$-qi embedding.
Thus $\alpha\subset N_D(f(A))$ where $D=kK+k+D_{\ref{stability}}(\delta,k,k)$.
Hence we can take $K_{\ref{qi map qc}}= kK+k+D_{\ref{stability}}(\delta,k,k)$.
\end{proof}
	
Given a pair of points in a quasiconvex set, existence of a geodesic joining them which is contained
in the set is not guaranteed. However, the following is true.
\begin{lemma}\label{dotted quasi-geodesic in quasi-convex set}
Given $ K\geq 0$ there is a constant 
$K_{\ref{dotted quasi-geodesic in quasi-convex set}}=K_{\ref{dotted quasi-geodesic in quasi-convex set}}( K)\geq 0$ such that the following holds:\\
Let $X$ be a (hyperbolic) metric space and let $Q$ be a $K$-quasiconvex subset of $X$. 
Then for all $x,y\in Q$, there exists a $K_{\ref{dotted quasi-geodesic in quasi-convex set}}$-quasigeodesic
in $X$ joining $x$ and $y$ whose image is contained in $Q$.
\end{lemma} 
The proof is very standard. Hence we just explain the overall idea of it.
If $\alpha:[a,b]\map X$ is a geodesic joining $x$ to $y$ then 
for all $t\in [a,b]$ one chooses $x_t\in Q$ such that $d(x_t,\alpha(t))\leq K$. Then $\beta:[a,b]\map X$
defined by $\beta(t)=x_t$ is a uniform quasigeodesic as required. 

The following lemma shows that finite union of quasiconvex sets is quasiconvex.
The proof is clear by induction and hence we skip it.
\begin{lemma}\label{finite qc}
Given $\delta\geq 0, k\geq 0$ and $n\in \NN$, there is a constant
$D_{\ref{finite qc}}=D_{\ref{finite qc}}(\delta, k,n)$ such that the following holds:

Suppose $X$ is a $\delta$-hyperbolic metric space and $\{A_i\}_{1\leq i\leq n}$ is a collection of $k$-quasiconvex 
subsets of $X$ such that $A_i\cap A_{i+1}\neq \emptyset$ for all $1\leq i\leq n-1$.
Then $\cup_i A_i$ is a $D_{\ref{finite qc}}$-quasiconvex set in $X$.
\end{lemma}

In general arbitrary union of quasiconvex sets need not be quasiconvex. However the following is true.
\begin{lemma}\label{quasi-convex union}
Given $\delta\geq 0, K\geq 0$ there is  a constant
$D_{\ref{quasi-convex union}}=D_{\ref{quasi-convex union}}(\delta, K)\geq 0$ such that the following holds:

Suppose $X$ is a $\delta$-hyperbolic metric space, $\{A_i\}$ is any (finite or infinite) sequence of $K$-quasiconvex sets in $X$
and $\gamma\subset X$ is a geodesic such that $A_i\cap A_{i+1}\cap \gamma\neq \emptyset$ for all $i\geq 1$.
Then $\cup_i A_i$ is a $D_{\ref{quasi-convex union}}$-quasiconvex set in $X$.
\end{lemma} 
\proof Let $x_i\in A_i\cap A_{i+1}\cap \gamma$ for all $i$. For all $i\leq j$, let $[x_i,x_j]$ denote the segment of $\gamma$
from $x_i$ to $x_j$. Clearly $[x_i, x_{i+1}]\subset N_K(A_{i+1})$ for all $i$ and hence 
$[x_i,x_j]\subset N_K(\cup_{i+1\leq k\leq j+1} A_k)\subset N_K(\cup_k A_k)$ for all $i\leq j$.
Now, given $x \in A_i, y\in A_j$, $i\leq j$ we have 
$[x,y]\subset N_{2\delta}([x,x_i]\cup [x_i,x_j]\cup [x_j,y])$ since clearly geodesic quadrilaterals
in a $\delta$-hyperbolic metric space are $2\delta$-slim by Lemma \ref{slim polygon}.
Hence, $[x,y]\subset N_{2\delta+K}(\cup_k A_k)$. Hence, we may choose 
$D_{\ref{quasi-convex union}}=2\delta+K$.
\qed

\smallskip
\subsection{Gromov boundary}
Now, we briefly recall some basic facts about Gromov boundary of hyperbolic spaces. 
For more details see \cite{GhH},\cite{bridson-haefliger}. 
Let $X$ be a hyperbolic geodesic metric space. Two (quasi)geodesic rays $\alpha,\beta$ are said to be
{\em asymptotic} if the Hausdorff distance between $\alpha$ and $\beta$ is finite. This gives an equivalence 
relation on the set $Geo(X)$ of all geodesic rays (resp. on the set $QGeo(X)$ of all quasigeodesics) in $X$. 
The equivalence class of a (quasi)geodesic ray $\alpha$ is denoted by $\alpha(\infty)$. We denote the set of 
all equivalence classes of (quasi)geodesic rays by $\partial X$ (resp. $\partial_q X$) and call them the {\em geodesic}
(resp. {\em quasigeodesic}) {\em boundary} of $X$.
If $X$ is a proper hyperbolic space then $\bar{X}:=X\cup \partial X$ is a compact metrizable space for 
a natural topology, in which $X$ is an open subset of $\bar{X}$.
However, since the spaces that we consider later could also be non-proper we shall use the following definition.

We recall that for any metric space $Z$ and points $z, z_1, z_2\in Z$, the Gromov product of $z_1,z_2$ with respect to
$z$ is the number $\frac{1}{2}(d(z,z_1)+d(z,z_2)-d(z_1,z_2))$ and it is denoted by $(z_1.z_2)_z$.

\begin{defn}\cite[Definition 4.1]{short-group}
Suppose $X$ is a hyperbolic metric space. A sequence of points $\{x_n\}$ in $X$ is said to converge
to infinity if $\lim_{m,n\map \infty} (x_m.x_n)_x=\infty$ for some (any) $x\in X$. 
\end{defn}
Let $S_{\infty}(X)$ denote the set of all sequence of points in $X$ which converge to infinity. On this set one defines
an equivalence relation by setting $\{x_n\}\sim \{y_n\}$ if and only if $\lim_{m,n\map \infty} (x_m.y_n)_x=\infty$. 
The following is a very basic lemma. See \cite{short-group} for a proof.

\begin{lemma}\label{lemma: seq bdry 1}
(1) If $\{x_n\}\in S_{\infty}(X)$ and $\{x_{n_k}\}$ is a subsequence of $\{x_{n_k}\}$ then $\{x_{n_k}\}\in S_{\infty}(X)$
and $\{x_n\}\sim \{x_{n_k}\}$.

(2) If $\{x_n\}, \{y_n\}\in S_{\infty}(X)$ then $\{x_n\}\sim \{y_n\}$ if and only if $\lim_{n\map \infty} (x_n.y_n)_x=\infty $.
\end{lemma}

\begin{defn}[{\bf Sequential boundary}]
The sequential boundary $\partial_s X$ of a hyperbolic metric space $X$ is defined to be $S_{\infty}(X)/\sim$.
\end{defn}
If $\{x_n\}\in S_{\infty}(X)$ then the equivalence class of $\{x_n\}$ will be denoted by $[\{x_n\}]$.
If $\xi=[\{x_n\}]\in \partial_s X$ then we say $x_n$ converges to $\xi$.
Here are some of the basic facts about boundaries of hyperbolic spaces. For the  parts (1) and (3) of the Lemma below see
\cite[Chapter III.H]{bridson-haefliger}, and for the part (2) see \cite[Lemma 2.4]{mahan-sardar}.

\begin{lemma}\label{bdry basic prop}
	Suppose $X$ is a hyperbolic metric space.
\begin{enumerate}
\item 
Given a quasigeodesic $\alpha:[0,\infty)\map X$, the sequence $\{\alpha(n)\}$ converges to infinity.
{\em (We denote the equivalence class of $\{\alpha(n)\}$ also by $\alpha(\infty)$. 
We say $\alpha$ joins $\alpha(0)$ to $\alpha(\infty)$.)}
This gives rise to an injective map $\partial_q X\map \partial_s X$.
\item Given $\delta\geq 0$ there is a constant 
$k_{\ref{bdry basic prop}}=k_{\ref{bdry basic prop}}(\delta)$ depending only on $\delta$ such that
the following hold:

If $X$ is $\delta$-hyperbolic metric space then (i) for any $x_0\in X$ and any $\xi\in \partial_s X$ 
there is a $k_{\ref{bdry basic prop}}$-quasigeodesic ray in $X$ joining $x_0$ to $\xi$. 
In particular, the map $\partial_q X\map \partial_s X$ mentioned above is surjective.

(ii) For all $\xi_1\neq \xi_2\in \partial_s X$ there is a 
$k_{\ref{bdry basic prop}}$-quasigeodesic line $\gamma$ in $X$ joining $\xi_1,\xi_2$, i.e. such that
$\xi_1$ is the equivalence class of $\{\gamma(-n)\}$ and $\xi_2$  is the equivalence class of $\{\gamma(n)\}$.
\item If $X, Y$ are hyperbolic metric spaces and $f:Y\map X$ is a qi embedding then $f$ induces an injective map
$\partial f:\partial_s Y\map \partial_s X$. This map is functorial:

(i) If $\iota:X\map X$ is the identity map then $\partial \iota$ is the identity map on $\partial_s X$.

(ii) If $g:Z\map Y$ and $f:Y\map X$ are qi embeddings of hyperbolic metric spaces then $\partial f\circ \partial g=\partial (f\circ g)$.
\end{enumerate}
\end{lemma}

\smallskip

{\bf Topology on $X\cup \partial_s X$.}

\smallskip
There is a natural way to put a Hausdorff topology on $X\cup \partial_s X$. The reader is referred to \cite[Definition 4.7]{short-group}
for details. We shall only include the following basic facts that we are going to need later.

(1) (see \cite[Lemma 4.6(2)]{short-group})
If $\{x_n\}$ is a sequence in $X$ and $\xi\in \partial_s X$ then  $x_n\map \xi$ if and only if $\{x_n\}$ converges to infinity
and $\xi=[\{x_n\}]$.

(2) (see \cite[Lemma 4.6(4)]{short-group}) If $\{\xi_n\}$ is a sequence in $\partial_s X$ and $\xi\in \partial_s X$. Then $\xi_n\map \xi$ if and only if the following holds:

Suppose $\xi_k$ is the equivalence class of $\{ x^k_n\}$ and $\xi$ is the equivalence class of $\{y_n\}$ then
$\lim_{k\map \infty} (\liminf_{m,n\map \infty}(x^k_m.y_n)_x)=\infty$ for any $x\in X$.

\begin{comment}
One can easily check using the results in \cite[Chapter 4]{short-group} that the following defines closed sets for the topology
mentioned there.

\begin{defn}
Suppose $A\subset \partial_s X$. We say that $A$ is a closed set if for any sequence $\xi_n$ in $\partial_s X$, 
$\xi_n\map \xi\in \partial_s X$ implies $\xi\in A$.
\end{defn}

\RED{State necessary lemmas.}
\end{comment}

The following lemma gives a geometric criteria for convergence and is well known among experts.
See \cite[Lemma 2.41]{mbdl2} for a proof.
\begin{lemma}\label{convg criteria}
For all $\delta\geq 0$ and $k\geq 1$ the following holds:\\
Suppose $\{x_n\}$, $\{y_n\}$ are sequences in a $\delta$-hyperbolic metric space $X$ and $\xi\in \partial_s X$.
Let $\alpha_{m,n}$ be a $k$-quasigeodesic joining $x_m, x_n$, and let $\beta_{m,n}$ be a $k$-quasigeodesic 
joining $x_m,y_n$ for all $m,n\in \NN$. Let $\gamma_n$ be a $k$-quasigeodesic ray
joining $x_n$ to $\xi$ for all $n\in \NN$. Let $x_0\in X$ be an arbitrary fixed point. Then

(0) $\{x_n\}$ converges to infinity if and only if $\lim_{m,n\map \infty} d(x_0, \alpha_{m,n})=\infty$.

(1) If $\{x_n\}, \{y_n\}$ both converge to infinity then $\{x_n\}\sim \{y_n\}$ if and only if
$\lim_{m,n\map \infty} d(x_0, \beta_{m,n})=\infty$.

(2) $x_n\map \xi$ if and only if $\lim_{n\map \infty} d(x_0, \gamma_n)=\infty$.
\end{lemma}

The following lemma gives yet another criteria for convergence.

\begin{lemma}\label{convg criteria 2}
For all $\delta\geq 0$ and $k\geq 1$ there is a constant $D_{\ref{convg criteria 2}}=D_{\ref{convg criteria 2}}(\delta,k)$ 
such that following holds:\\
Suppose $X$ is a $\delta$-hyperbolic metric space, $x_0\in X$, $\xi\in \partial_s X$ and
$\{x_n\}$ is a sequence of points in $\bar{X}$. Suppose $\alpha$ is a $k$-quasigeodesic ray
converging to $\xi$ and $\alpha_n$ is a $k$-quasigeodesic segment joining $x_0$ to $x_n$. 
Then $x_n\map \xi$ if and only if for all $R\geq 0$ there is $N\in \NN$ such that 
$Hd(\alpha\cap B(x_0; R), \alpha_n\cap  B(x_0; R))\leq D_{\ref{convg criteria 2}}+d(x_0,\alpha(0))$ for all $n\geq N$.
\end{lemma}
Informally we shall refer to the conclusion of this lemma by saying that 
$x_n\map \xi$ if {\bf $\alpha_n$'s fellow travel $\alpha$ for a longer and longer time} as 
$n\map \infty$. The idea of the proof is very similar to that of Lemma 1.15 and also Lemma 3.3 of 
\cite[Chapter III.H]{bridson-haefliger}. Since this is very standard we skip it. One is also referred
to \cite[Lemma 2.40, Lemma 2.45(2)]{mbdl2}.

The following lemma gives one of the main tools to prove the main theorem of this paper.
\begin{lemma}\label{diagonal sequence}
Suppose $X$ is a $\delta$-hyperbolic metric space and $\xi\in \partial_s X$. Suppose 
for all $n\in \NN$ there is a sequence $\{x^n_k\}$ in $X$ with $x^n_k\map \xi$
as $k\map \infty$. Then there is a subsequence $\{m_n\}$ of the
sequence of natural numbers such that $x^n_{m_n}\map \xi$.
\end{lemma} 
\proof Let $k_0=k_{\ref{bdry basic prop}}(\delta)$.
By Lemma \ref{bdry basic prop}(2) for all $x\in X$ and 
$\eta\in \partial_s X$ there is a $k_0$-quasigeodesic ray joining $x$ to $\eta$. Let $x_0\in X$. 
Let $\gamma_{n,k}$ be a $k_0$-quasigeodesic ray joining $x^n_k$ to $\xi$ for all $n, k\in \NN$.
For all $n\in \NN$, $d(x_0,\gamma_{n,k})\map \infty$ as $k\map \infty$ by Lemma \ref{convg criteria}(3). 
Thus for all $n\in \NN$
we can find $m_n\in \NN$ such that $d(x_0, \gamma_{n,i})>n$ for all $i\geq m_n$.
Clearly, $x^n_{m_n}\map \xi$ by Lemma \ref{convg criteria}(3). \qed

\medskip
{\bf Limit sets}

\begin{defn}{\em (Limit set)}
Suppose $X$ is a hyperbolic metric space and $A\subset X$. Then the {\em limit set of $A$} in $\partial_s X$
is the set $\Lambda(A)$ of all points $\xi \in \partial_s X$ such that there is a sequence $\{a_n\}$ in $A$ converging
to $\xi$.
\end{defn}
The following lemma is easy. However, we include a proof for the sake of completeness.

\begin{lemma}\label{ray in qc}
Given $\delta\geq 0$ and $k\geq 0$ there is a constant $K_{\ref{ray in qc}}=K_{\ref{ray in qc}}(\delta, k)$ such that the following holds:\\
Suppose $X$ is a $\delta$-hyperbolic metric space and $A\subset X$ is a $k$-quasiconvex subset. Then for all
$\xi\in \Lambda(A)$ and $x\in A$ there is a $K_{\ref{ray in qc}}$-quasigeodesic ray $\gamma$ of $X$ contained in $A$ 
where $\gamma$ joins $x$ to $\xi$.
\end{lemma}
\proof Let $k_0=k_{\ref{bdry basic prop}}(\delta)$. 
We know by Lemma \ref{bdry basic prop}(2) that there is a $k_0$-quasigeodesic ray 
$\alpha:[0,\infty)\map X$ joining $x$ to $\xi$.
Now, since $\xi\in \Lambda(A)$ there is a sequence $\{x_n\}$ in $A$ such that $\lim_{n\map \infty}(x_n.\alpha(n))_x=\infty$.
Given any $p\in \alpha$ there is $N\in \NN$ such that 
$(x_n.\alpha(n))_x\geq d(x,p)+D_{\ref{stability}}(\delta, k_0,k_0)+k+\delta$ for all $n\geq N$.
Now, by stability of quasigeodesics there is a point $q\in [x,\alpha(N)]$ such that 
$d(p,q)\leq D_{\ref{stability}}(\delta, k_0,k_0)$. We look at the geodesic triangle $\bigtriangleup xx_N\alpha(N)$. 
Since $\bigtriangleup xx_N\alpha(N)$
is $\delta$-slim $q\in N_{\delta}([x,x_N]\cup [x_N, \alpha(N)]$. However if there is $q'\in [x_N, \alpha(N)]$
with $d(q,q')\leq \delta$ then 
$2(x_N\cdot\alpha(N))_x=d(x,x_N)+d(x,\alpha(N))-d(x_N,\alpha(N))\leq (d(x,q)+d(q,q')+d(q',x_N))+
	(d(x,q)+d(q,q')+d(q',\alpha(N)))-d(x_N,\alpha(N))
	=2d(x,q)+2d(q,q')\leq 2d(x,q)+2\delta.$
Thus $(x_N\cdot \alpha(N))_x\leq d(x,q)+\delta\leq d(x,p)+D_{\ref{stability}}(\delta, k_0,k_0)+\delta$ contradicting
the choice of $N$.
Hence, $q\in N_{\delta}([x,x_N])$. Finally, since $A$ is $k$-quasiconvex, $q\in N_{\delta+k}(A)$.
Hence, $p$ is contained in the $(D_{\ref{stability}}(\delta, k_0,k_0)+k+\delta)$-neighborhood of $A$.
Let $D=D_{\ref{stability}}(\delta, k_0,k_0)+k+\delta$. Then for all $t\in [0,\infty)$ there is a point $x_t\in A$
such that $d(x_t, \alpha(t))\leq D$. Clearly $t\mapsto x_t$ is a $(k_0+2D)$-quasigeodesic ray as required. 
Thus by choosing $K_{\ref{ray in qc}}=k_{0}+2D$ we are done. \qed


\begin{lemma}\label{limit set intersection}
Suppose $A$ is a (closed) subset of a hyperbolic metric space $X$ and $\gamma:[0,\infty)\map X$ is a quasigeodesic ray
in $X$. Suppose $a_n$ is a nearest point projection of $\gamma(n)$ on $A$ for all $n\in \NN$.
If the set $\{a_n\}$ is unbounded then $\gamma(\infty)\in \Lambda(A)$.

The converse is true if $A$ is quasiconvex in which case $\gamma$ is contained in a finite neighborhood of
$A$.
\end{lemma}
\proof 
Let $b_n=\gamma(n)$. Let $x\in A$. Without loss of generality we shall assume that $d(x,a_n)\map \infty$.
We claim that $d(x, [a_n,b_n])\map \infty$. Suppose not.
Then there is $R\geq 0$ such that $d(x,[a_n,b_n])\leq R$ for infinitely many $n$. For any such
$n$, let $x_n\in [a_n,b_n]$ such that $d(x,x_n)\leq R$. Since $a_n$ is a nearest point projection of
$b_n$ on $A$ and $x\in A$, we must have $d(x_n,a_n)\leq d(x_n,x)\leq R$. It follows that $d(x,a_n)\leq 2R$.
This gives a contradiction to the assumption that $d(x,a_n)\map \infty$ and completes the proof. 

For the converse, suppose $X$ is $\delta$-hyperbolic and $A$ is $k$-quasiconvex in $X$. Let $x\in A$. Then by Lemma \ref{ray in qc} 
there exists a $K_{\ref{ray in qc}}(\delta, k)$-quasigeodesic ray, say $\alpha$, of $X$ contained in $A$ which join 
$x$ to $\gamma(\infty)$. Since $\{\alpha(n)\}\sim \{\gamma(n)\}$, it easily follows from Lemma
\ref{convg criteria}(2) coupled with stability of quasigeodesics that $Hd(\alpha,\gamma)<\infty$.
Thus $\gamma$ is contained in a finite neighborhood of $A$. Thus the sequence $\{a_n\}$ is unbounded. \qed


\subsection{Cannon-Thurston maps}\label{section: CT maps}
The following proposition is very standard.

\begin{prop}\textup{(\cite[Theorem 5.38]{vaisala}, \cite[Theorem 7.2.H]{gromov-hypgps}, \cite[III.H, Theorem 3.9]{bridson-haefliger})}\label{ct example1}
If $f:X\map Y$ is a qi embedding where $X,Y$ are hyperbolic spaces then the
map $\partial f:\partial X\map \partial Y$, as defined in Lemma \ref{bdry basic prop}(3)
is continuous. Moreover, if $X,Y$ are proper metric spaces then it is a closed
embedding.

If $f$ is quasiisometry then $\partial f$ is a homeomorphism.
\end{prop}

Consequently it is natural to ask if non-qi embeddings could also induce continuous maps
between the Gromov boundaries of hyperbolic spaces. This provides a motivation
to the following definition. 
\begin{defn}[{\bf Cannon-Thurston map}] \label{defn: CT} 
Let $f:X\rightarrow Y$ be a map between two hyperbolic metric spaces. We say that {\em the
Cannon-Thurston (CT) map} exists for $f$ or $f$ admits {\em the CT map}
if the following hold:
	
(1) Given any $\xi\in \partial X$ and any sequence $\{x_n\}$ in $X$ with $\xi=[\{x_n\}]$, 
the sequence $\{f(x_n)\}\in S_{\infty}(Y)$ and $[\{f(x_n)\}]$ depends only $\xi$ and not on
on the sequence $\{x_n\}$. Thus we have a map $\partial f:\partial X\map \partial Y$.

(2) The map $\partial f$ is continuous.
\end{defn}
As mentioned in the introduction, the notion of CT maps between hyperbolic metric spaces were introduced 
by Mahan Mitra (\cite{mitra-trees}). Our definition is an adaptation from his definition; however,
it is easy to check that these definitions are equivalent. One also notes that the CT map is unique if it exists.
However, the Lemma below shows that (2) in the definition of CT maps follows from (1).

\begin{lemma}\label{new CT lemma 1}
Suppose $f:X\map Y$ is any map between hyperbolic metric spaces which satisfies the condition (1) of Definition \ref{defn: CT}.
Then $\partial f:\partial X\map \partial Y$ is continuous too, i.e. $\partial f$ is the CT map.
\end{lemma}
\proof Fix $x_0\in X, y_0\in Y$.  Suppose $\{\xi_n\}$ is a sequence in $\partial_s X$ and $\xi_n\map \xi\in \partial_s X$.
We want to show that $\partial f(\xi_n)\map \partial f (\xi)$. Suppose this is not the case.
Let $\xi_k=[\{x^k_n\}]$ and $\xi=[\{x_n\}]$. Then there is $R\geq 0$ such that upto passing to a subsequence
of $\{\xi_n\}$ we may assume that $\liminf_{m,n\map \infty}(f(x^k_m).f(x_n))_{y_0}\leq R$ and
$\liminf_{m,n\map \infty}(x^k_m.x_n)_{x_0}\geq k$ for all $k$. This implies that for all $k\in \NN$ there is $m_k\in \NN$
such that $(x^k_{m_k}.x_{m_k})_{x_0}\geq k$ but $(f(x^k_{m_k}).f(x_{m_k}))_{y_0}\leq R$.
Consequently, by Lemma \ref{lemma: seq bdry 1} $x^k_{m_k}\map \xi$ but $f(x^k_{m_k})\not \map \partial f(\xi)$- a contradiction. \qed

\begin{cor}\label{new CT lemma}
Suppose $X, Y$ are hyperbolic metric space, and  $f:X\map Y$ and $g:\partial_s X\map \partial_s Y$
are any maps which satisfy the following property:\\
For any $\xi\in \partial_s X$ and any sequence $x_n\map \xi$ where $x_n\in X$ for all $n\in \NN$ there is a
subsequence $\{x_{n_k}\}$ of $\{x_n\}$ such that $f(x_{n_k})\map g(\xi)$ as $k\map \infty$. Then $g$ is the
CT map induced by $f$.
\end{cor}
\proof Suppose $\xi\in \partial_s X$ and $\{x_n\}$ is a sequence in $X$ converging to $\xi$.
Then any subsequence of $\{f(x_n)\}$ has a subsequence which converges to $g(\xi)$. Since $\bar{Y}$
is a Hausdorff space it follows by Lemma \ref{subseq} that $f(x_n)\map g(\xi)$. Then we are done by Lemma 
\ref{new CT lemma 1}. \qed

\begin{rem}{\bf When CT map does not exist:}
Suppose $X, Y$ are hyperbolic metric spaces and $f:X\map Y$ is any map. We note that the CT map does not exist for $f$ means
that there is a point $\xi\in \partial_s X$ and a sequence $x_n\map \xi$ in $X$ such that $\{f(x_n)\}$  does not converge
to any point of $\partial_s Y$. This in turn implies that either (1) there are two subsequences of $\{f(x_n)\}$ converging
to two distinct points of $\partial_s Y$ or (2) there is subsequence of $\{f(x_n)\}$ which has no subsequence
converging to a point of $\partial_s Y$. For instance (2) holds if $\{f(x_n)\}$ is bounded, which,
of course, is impossible if $f$ is a proper embedding.
\end{rem}

The following lemma gives a sufficient condition for the existence of Cannon-Thurston maps.
Although it was proved by Mitra for proper hyperbolic metric spaces, it is true
for any hyperbolic metric space. In fact, it is immediate from Lemma \ref{convg criteria}(0)
and Lemma \ref{new CT lemma 1}.
\begin{lemma} {\em ({\bf Mitra's criterion},\textup{\cite[Lemma 2.1]{mitra-trees})}} \label{mitra's criterion}
	Suppose $X,Y$ are hyperbolic geodesic metric spaces and $f:Y\rightarrow X$ is a proper embedding. 
	Then $f$ admits the Cannon-Thurston map if the following holds:
	
	Given $y_0\in Y$ there exists a non-negative function $M(N)$, such that $M(N)\rightarrow \infty $ as $N\rightarrow \infty$ and 
	such that for all geodesic segments $\lambda$ lying outside $B(y_0,N)$ in $Y$, any geodesic segment in $X$ joining the end points of 
	$f(\lambda)$ lies outside $B(f(y_0),M(N))$ in $X$. 
\end{lemma}

\begin{rem}\label{CT iff Mitra}
\begin{enumerate}
\item In Mitra's criterion properness of $f$ is redundant since it follows from the rest.
\item It is clear from the proof of \cite[Lemma 2.1]{mitra-trees} that if $Y$ is a proper metric space then
the CT map exists for a proper embedding $f:Y\map X$ if and only if $f$ satisfies Mitra's criterion.
\end{enumerate}

\end{rem}

To end this section we include a few things about Gromov hyperbolic groups.
\begin{defn}
If $G$ is a hyperbolic group then the {\em Gromov boundary} $\partial G$ of $G$ is defined 
to be the boundary of any of its Cayley graphs with respect to a finite generating set.
\end{defn} 
Since for different finite generating sets the corresponding Cayley graphs of $G$ are quasiisometric, 
$\partial G$ is well defined by Proposition \ref{ct example1}.

The existence of CT maps gives the following useful criteria for quasiconvexity. 
\begin{lemma}\label{inj of ct}(\cite[Lemma 2.1]{mj-scott-thm})
If $G$ is a hyperbolic group and if $H$ is a hyperbolic subgroup of $G$
such that the inclusion $H\map G$ admits the CT map $\partial H\rightarrow \partial G$
then $H$ is quasiconvex in $G$ if and only if the CT map is injective.
\end{lemma}


\smallskip



\section{Electric geometry and CT maps}\label{3}

\subsection{Farb's electrified spaces}\label{section electric 1}

Given a geodesic metric space $X$ and a collection of its subsets $\{A_i\}_{i\in I}$, Farb (see \cite{farb-relhyp})
defined the electrified or coned-off
space $\HAT{X}$ of $X$ with respect to $\{A_i\}$ to be a new metric space as follows. For 
each $A_i$ one introduces an extra point $c_i$ called the {\em cone point} for $A_i$. Then 
each point of $A_i$ is joined to $c_i$ by an edge of length $1$. In other words, as a set one has 
$\HAT{X}=X\sqcup (\cup_{i\in I}  A_i\times [0,1])\sqcup \{c_i\}_{i\in I}/\sim$ 
where $(x,1)\sim c_i$ for all $x\in A_i$ and $i\in I$
and $(x,0)\sim x$ for all $x\in A_i$ and $i\in I$. 
For all $A_i$ and all $x,x'\in A_i$, concatenation of the edge joining $x$ to $c_i$ and the edge
joining $c_i$ to $x'$ is a path of length $2$ from $x$ to $x'$ in $\HAT{X}$. 
We shall call this path the {\bf electric path} joining $x,x'$ and denote it by $e_i(x,x')$.
On the coned-off space $\HAT{X}$ one may then put the natural length metric using these paths. 
We shall assume that these are geodesic metric spaces. This is true in all the examples involving groups. 
We record the following basic lemma that we shall need later. See e.g. \cite[Proposition 3.1]{farb-relhyp}
for an idea of the proof.

\begin{comment}

\begin{lemma}
Suppose $f:X\map Z$ is a coarsely Lipschitz, coarsely surjective map. For all $z\in f(X)$, let $A_z=f^{-1}(z)$.
Let $\HAT{X}$ be the coned-off space obtained from $X$ by coning the $A_z$'s. Suppose there are constants
$D_0>0, D_1>0$ such that for all $z_1,z_2\in f(X)$ with $d_Z(z_1,z_2)\leq D_0$ there are $x_i\in A_{z_i}$,
$i=1,2$ with $d_X(x_1,x_2)\leq D_1$. Suppose $\HAT{f}:\HAT{X}\map Z$ be extension of $f$ obtained by mapping
line segments joining the cone point for $A_z$ and each point of $A_z$ to $z$ for all $z\in f(X)$.
Then $\HAT{f}$ is a uniform quasiisometry.
\end{lemma}

\begin{lemma}
Given $D\geq 0$ there is $K\geq 1$ such that the follwing holds:\\
Suppose $X$ is a geodesic metric space and $\{A_i\}$ is a collection of subsets of $X$ each of which has diameter
at most $D$. Then the inclusion $X\map \HAT{X}$ is a $K$-quasiisometry.
\end{lemma}
\end{comment}

\begin{lemma}\label{basic cone-off 3}
Given $D\geq 0$ there is $K_{\ref{basic cone-off 3}}=K_{\ref{basic cone-off 3}}(D)$ such that the following holds:

Suppose $X$ is a geodesic metric space and $\{A_i\}_{i\in I}$ and $\{B_i\}_{i\in I}$ are two sets of subsets
of $X$ indexed by the same set $I$. Suppose $\HAT{X}_A, \HAT{X}_B$ are the coned-off spaces obtained from coning the
$A_i$'s and $B_i$'s respectively. 
Let $\phi: \HAT{X}_A \map \HAT{X}_B$ be the extension of the identity map $X\map X$
obtained by mapping the open ball of radius $1$ about the cone point for $A_i$ to the cone point for $B_i$.

If there is $D\geq 0$ such that  $Hd(A_i, B_i)\leq D$  for all $i\in I$ then $\phi$ is a $K_{\ref{basic cone-off 3}}$-quasiisometry.
\end{lemma}

\begin{conv}
We fix the following notation and convention for the subsections \ref{section electric 1}, \ref{section 3.2} and 
\ref{section main thm}. We shall assume that the metric space $X$  
is $\delta_0$-hyperbolic and that we are coning off $k_0$-quasiconvex sets $\{A_i\}$ in $X$. 
We note that by Lemma \ref{dotted quasi-geodesic in quasi-convex set}
for all $A_i$ and all $x,x'\in A_i$ there is a uniform
quasigeodesic of $X$ joining $x,x'$ whose image is contained in $A_i$.
We shall assume that these are all $\lambda_0$-quasigeodesics. In these subsections we shall suppress the 
explicit dependence of the functions or constants appearing in various propositions, lemmas,
and corollaries on $\delta_0, k_0$ and $\lambda_0$.
\end{conv}

\begin{defn}{\em ({\bf Concatenation of paths})}
(1) If $\alpha_1:I_1\map X$ and $\alpha_2:I_2\map X$ are any two maps where $I_1, I_2$
are intervals in $\RR$ such that $\max I_1, \min I_2$ exists and $\alpha_1(\max I_1)=\alpha_2(\min I_2)$
then we define the {\em concatenation} of $\alpha_1, \alpha_2$ to be the map $\alpha: I\map X$ 
where $I=I_1\cup \{t+\max I_1-\min I_2: t\in I_2\}$ and $\alpha|_{I_1}=\alpha_1$
and $\alpha(t+\max I_1-\min I_2)=\alpha_2(t)$ for all $t\in I_2$.

We shall denote the concatenation of $\alpha_1$ and $\alpha_2$ by $\alpha_1*\alpha_2$.

(2) {\em ({\bf De-electrification})} Suppose $\gamma$ is a continuous path in $\HAT{X}$ with end points in $X$. 
A de-electrification $\gamma^{de}$ of $\gamma$ is a path in $X$ constructed 
from $\gamma$ as follows. If $\gamma$ is the concatenation 
$\gamma_1*e_{i_1}(x_1,x'_1)*\gamma_2*\cdots *e_{i_n}(x_n,x'_n)*\gamma_{n+1}$
where $\gamma_i$'s are paths in $X$ and $e_{i_j}(x_j,x'_j)$ are electric paths
joining $x_j, x'_j\in A_{i_j}$ and passing through the cone point $c_{i_j}$ then 
$\gamma^{de}$ is the concatenation 
$\gamma_1*c(x_1,x'_1)*\gamma_2*\cdots *c(x_n,x'_n)*\gamma_{n+1}$ where each 
$c(x_j,x'_j)$ is a $\lambda_0$-quasigeodesic segments of $X$ joining $x_j,x'_j$ 
which is contained in the quasiconvex set $A_{i_j}$. 
\end{defn}
We refer the reader to \cite[Section 2.5.2]{mj-dahmani} for a slightly different definition
of de-electrification of paths.
We note that one may define de-electrification of paths that pass through
infinitely may cone points also in an obvious manner. Of course we have not defined
electrification (see \cite[Section 3.3]{farb-relhyp}) of paths of $X$ since we do 
not need them. The following lemma explains 
the significance of de-electrification of paths in our context.

\begin{lemma}\label{electric qc} 
Given $l\in \mathbb Z_{\geq 0}$ there is a constant 
$K_{\ref{electric qc}}=K_{\ref{electric qc}}(l)$
such that the following holds:\\
	Let $\gamma$ be a geodesic of length at most $l$ in $\HAT{X}$. Then the de-electrification of $\gamma$ is a
$K_{\ref{electric qc}}$-quasiconvex path in $X$.
\end{lemma}
\begin{proof} We note that $\gamma$ is the concatenation of at most $l$ electric paths
and $l+1$ geodesic segments of $X$. Thus a de-electrification $\gamma^{de}$ of $\gamma$ is the 
concatenation of at most $l+1$ geodesic segments of $X$ and at most $l$ number of $\lambda_0$-quasigeodesic
segments in $X$. The lemma then follows from Lemma \ref{qc lemma 1} and Lemma \ref{finite qc}. In fact one 
may take $K_{\ref{electric qc}}=D_{\ref{finite qc}}(\delta_0,\lambda_0,l)$.
\end{proof}

The following result is motivated from Farb's (\cite{farb-relhyp}) weak relative hyperbolicity. The statement was known to be 
true to the specialists for a long time. However, the first rigorous proof of it appears in \cite{mj-dahmani}. 
See also \cite[Proposition 2.6]{kapovich-rafi}.

\begin{prop}\textup{(\cite[Proposition 2.10]{mj-dahmani})}\label{mj-dahmani}
Given $\delta\geq 0,C\geq 0$ there exists a constant 
$\delta_{\ref{mj-dahmani}}= \delta_{\ref{mj-dahmani}}(\delta, C)\geq 0$ such that 
if $X$ is a $\delta$-hyperbolic metric space and $\{A_i\}$ is a collection of $C$-quasiconvex 
subsets of $X$ then $\HAT{X}$ is $\delta_{\ref{mj-dahmani}}$-hyperbolic.
\end{prop}

In addition to the above proposition the following nice result appears in \cite{mj-dahmani} that motivates 
the rest of the section.
\begin{prop}\textup{(\cite[Proposition 2.11]{mj-dahmani})}\label{qc hat qc}
Given $\delta\geq 0,C\geq 0, K\geq 0$ there exists a constant 
$K_{\ref{qc hat qc}}= K_{\ref{qc hat qc}}(\delta, C, K)$ such that the following holds:\\
Suppose we have the hypotheses of Proposition \ref{mj-dahmani}. Then any 
$K$-quasiconvex subset $Q$ of $X$ is $K_{\ref{qc hat qc}}$-quasiconvex in $\HAT{X}$.
\end{prop}

A form of a converse of this proposition also appears in \cite[Proposition 2.12]{mj-dahmani}. However, this leads us to
the following question. Are there analogues of the above proposition where quasiconvexity is replaced by 
weaker conditions? The main technical result of the paper and the main result
of this section will try to formalize these questions and will obtain some answers to them. We start with some basic lemmas
about geodesics in $\HAT{X}$. The following lemma is taken from \cite{kapovich-rafi}. However, we give
an alternative proof to it here.

\begin{lemma}\textup{(\cite[Corollary 2.4]{kapovich-rafi})}\label{kapovich-rafi main cor}
There is a constant 
$D_{\ref{kapovich-rafi main cor}}=D_{\ref{kapovich-rafi main cor}}(\delta_0, k_0,\lambda_0)$
such that the following holds:\\
Suppose $x,y\in X$. Suppose $\alpha$ is a geodesic in $X$ and  $\beta$ is a geodesic
in $\HAT{X}$ both joining $x,y$. Then $Hd_{\HAT{X}}(\alpha, \beta)\leq D_{\ref{kapovich-rafi main cor}}$.
\end{lemma} 
\proof By Proposition \ref{qc hat qc} $\alpha$ is a $K$-quasiconvex subset of $\HAT{X}$
where $K=K_{\ref{qc hat qc}}(\delta_0,k_0)$. Hence, by Lemma \ref{dotted quasi-geodesic in quasi-convex set} 
there is a $K_{\ref{dotted quasi-geodesic in quasi-convex set}}(K)$-quasigeodesic, say 
$\gamma:[a,b]\map \HAT{X}$, joining $x,y$ such that $\gamma([a,b])\subset \alpha$. Let 
$k= K_{\ref{dotted quasi-geodesic in quasi-convex set}}(K)$.
Now let $t_0=a<t_1<\cdots <t_{n-1}<t_n=b$ be points on $[a,b]$ such that $0<t_n-t_{n-1}\leq 1$ and
$t_{i+1}-t_i=1$ for all $i$, $0\leq i\leq n-2$. Let $x_i=\gamma(t_i)$.
Since for each $i$, $d_{\HAT{X}}(x_i,x_{i+1})\leq k|t_i-t_{i+1}|+k\leq 2k $, if $\beta_i$ is a geodesic in $\HAT{X}$ 
joining $x_i,x_{i+1}$ then any de-electrification of it, say $\beta^{de}_i$, is a  
$K_{\ref{electric qc}}(\delta_0,\lambda_0,2k)$-quasiconvex path in $X$
by Lemma \ref{electric qc}. Hence, once again by Lemma \ref{dotted quasi-geodesic in quasi-convex set} there is a 
$K_{\ref{dotted quasi-geodesic in quasi-convex set}}(K_{\ref{electric qc}}(\delta_0,\lambda_0,2k))$-quasigeodesic in 
$X$, say $\alpha_i$ joining $x_i, x_{i+1}$ contained in $\beta^{de}_i$. Let 
$k_1= K_{\ref{dotted quasi-geodesic in quasi-convex set}}(K_{\ref{electric qc}}(\delta_0,\lambda_0,2k))$.
Let $\alpha'$ be the  concatenation of the various $\alpha_i$'s and let $\beta'$ be the concatenation of the 
various $\beta_i$'s. 

Since $d_{\HAT{X}}(x_i,x_{i+1})\leq 2k$, the $\HAT{X}$-diameters of each $\beta_i$'s, $\beta^{de}_i$'s, and $\alpha_i$'s are
at most $2k+2$. It follows that $Hd_{\HAT{X}}(\beta', \gamma)$, and $Hd_{\HAT{X}}(\alpha', \gamma)$
are both at most $2k+2$. On the other hand, since $\alpha_i$'s are $k_1$-quasigeodesics in $X$,
by stability of quasigeodesics it is clear that $Hd_X(\alpha, \alpha')\leq D_{\ref{stability}}(\delta_0,k_1,k_1)$, whence 
$Hd_{\HAT{X}}(\alpha, \alpha') \leq D_{\ref{stability}}(\delta_0,k_1,k_1)$. Thus
$Hd_{\HAT{X}}(\alpha, \gamma)\leq Hd_{\HAT{X}}(\alpha, \alpha')+Hd_{\HAT{X}}(\alpha', \gamma)\leq 2k+2+D_{\ref{stability}}(\delta_0,k_1,k_1)$. 
Once again, since $\gamma$ is a $k$-quasigeodesic in $\HAT{X}$, by  stability of quasigeodesics 
we have that $Hd_{\HAT{X}}(\beta, \gamma)\leq D_{\ref{stability}}(\delta_{\ref{mj-dahmani}}(\delta_0,k_0),k,k)$. Hence, 
$Hd_{\HAT{X}}(\alpha, \beta)\leq Hd_{\HAT{X}}(\alpha, \gamma)+Hd_{\HAT{X}}(\gamma, \beta) \leq 
D_{\ref{kapovich-rafi main cor}}$ where 
$$ D_{\ref{kapovich-rafi main cor}}=2k+2+D_{\ref{stability}}(\delta_0,k_1,k_1)+D_{\ref{stability}}(\delta_{\ref{mj-dahmani}}(\delta_0,k_0),k,k).
\qed $$

\begin{cor}\label{electric cor2}
For all $D>0$ there is $D'>0$ where $D'\map \infty$ as $D\map \infty$ and such that the following holds:\\
Suppose $x_0, x, y\in X$ and $d_{\HAT{X}}(x_0, [x,y]_{\HAT{X}})\geq D$.
Then $d_X(x_0, [x,y]_X)\geq D'$
\end{cor}

\subsection{$\partial X$ vs $\partial \HAT{X}$}\label{section 3.2}
\smallskip
The next lemma follows in a straightforward way from \cite[Corollary 2.4]{kapovich-rafi}. Compare it with
\cite[Corollary 6.4]{abbott-manning}. It is an improvement on Lemma \ref{kapovich-rafi main cor}.
For the sake of completeness, we include a proof using Lemma \ref{kapovich-rafi main cor}.

\begin{lemma}\label{hor rays}
There are constants $K_{\ref{hor rays}}=K_{\ref{hor rays}}(\delta_0, k_0,\lambda_0)$,
and $D_{\ref{hor rays}}=D_{\ref{hor rays}}(\delta_0, k_0,\lambda_0)$
such that the following holds:\\
Suppose $\gamma$ is a geodesic ray in $X$ such that its image in $\HAT{X}$ is unbounded.
Then there is a $K_{\ref{hor rays}}$-quasigeodesic ray $\beta$ of $\HAT{X}$ such that 
$Hd_{\HAT{X}}(\beta, \gamma)\leq D_{\ref{hor rays}}$.
\end{lemma} 
\begin{proof}
Let $\gamma:[0,\infty)\map X$ be a geodesic ray in $X$ such that its image in 
$\HAT{X}$ is unbounded. Let $\gamma(0)=x_0$. Then for any $t\in [0,\infty)$ if 
$\alpha$ is a geodesic in $\HAT{X}$ joining $x_0$ to $\gamma(t)$ then by Lemma \ref{kapovich-rafi main cor}, $Hd_{\HAT{X}}(\gamma([0,t]),\alpha)\leq D_{\ref{kapovich-rafi main cor}}$. 
Now, let $(x_n)_{n\in\mathbb{N}}$ be a sequence of points on $\gamma$ such that
$d_{\HAT{X}}(x_0,x_n)>d_{\HAT{X}}(x_0, x_{n-1})+2D_{\ref{kapovich-rafi main cor}}$ for all $n\in \NN$.
Let $\alpha_{m,n}$ be a geodesic in $\HAT{X}$ joining $x_m$ and $x_n$ for all $m,n\in \NN$, $m<n$.
Let $y_{i,n}$ be a nearest point of $\alpha_{0,n}$ from $x_i$, $1\leq i\leq n$.
We note that $d_{\HAT{X}}(x_i,y_{i,n})\leq D_{\ref{kapovich-rafi main cor}}$, $1\leq i\leq n$. 
Hence for all $n\in \NN$ and $1\leq i\leq (n-1)$, we have
$d_{\HAT{X}}(x_0,y_{i+1,n})\geq d_{\HAT{X}}(x_0,x_{i+1})-D_{\ref{kapovich-rafi main cor}}> d_{\HAT{X}}(x_0,x_i)+2D_{\ref{kapovich-rafi main cor}}-D_{\ref{kapovich-rafi main cor}}
\geq d_{\HAT{X}}(x_0,y_i)$. Hence, it immediately follows from Lemma \ref{concat geod}
that the concatenation of the $\alpha_{n,n+1}$'s, $n\geq 0$ is a $K_{\ref{concat geod}}(D_{\ref{kapovich-rafi main cor}})$-quasigeodesic
in $\HAT{X}$. Let us call it $\beta$. Finally, since $Hd_{\HAT{X}}(\alpha_{n,n+1}, \gamma([n,n+1]))\leq D_{\ref{kapovich-rafi main cor}}$ for all $n\geq 0$ by Lemma \ref{kapovich-rafi main cor}, it follows that 
$Hd_{\HAT{X}}(\beta, \gamma)\leq D_{\ref{kapovich-rafi main cor}}$. Hence we may choose $K_{\ref{hor rays}}=K_{\ref{concat geod}}(D_{\ref{kapovich-rafi main cor}})$ and $D_{\ref{hor rays}}=D_{\ref{kapovich-rafi main cor}}$.
\end{proof}


\begin{conv}
Recall that $QGeod(X)$ denotes the asymptotic classes of all quasigeodesic rays in $X$. 
The following sets will be important for the rest of this subsection.

(1) $\partial_h X=\{\xi\in \partial X: \xi=\gamma(\infty), \gamma\in QGeod(X) 
\mbox{ is an unbounded subset of}\,\, \HAT{X}\}$.

(2) $\partial_v X=\cup_{i\in I} \Lambda(A_i)$ 

Intuitively we think of the quasigeodesic rays converging to points of $\partial_h X$ as 
the {\em horizontal} ones relative to the map $X\map \HAT{X}$ and those converging 
to points of $\partial_v X$ as {\em vertical}. We note that $\partial_v X\cap \partial_h X=\emptyset$
by Lemma \ref{limit set intersection}. Also $\partial_v X\cup \partial_h X\BLUE{\subset} \partial_s X$.
But this inclusion is not an equality in general. 
\end{conv}

However, Lemma \ref{hor rays} and Corollary \ref{electric cor2} immediately imply the following result
which was proved first in \cite[Theorem 3.2]{dowdall-taylor}. We include a sketch of proof for the
sake of completeness.

\begin{theorem}\textup{(\cite{dowdall-taylor})}\label{electric cor} \textup{(1)}  
We have a continuous map $\phi_X:\partial_h X\map \partial \HAT{X}$ such that the following holds: Suppose $\{x_n\}$ 
is a sequence in $X$. If $\{x_n\}$ converges to a point of $\xi\in \partial_h X$ then $\{x_n\}$ converges 
to $\phi_X(\xi)\in \partial \HAT{X}$.

\textup{(2)} There is a natural continuous inverse $\psi_X$ of $\phi_X $.
In particular, $\phi_X$ is a homeomorphism.
\end{theorem}

{\em Sketch of proof:} (1) Suppose $x_n\map\xi\in \partial_h X$. Let $\gamma$ be a quasigeodesic ray in $X$ with $\gamma(\infty)=\xi$
and $\gamma(0)=x_1$. Then $[x_1, x_n]_X$ fellow travels $\gamma$ in $X$ for longer and longer time 
as $n\map\infty$ by Lemma \ref{convg criteria 2}. It then follows from Lemma \ref{kapovich-rafi main cor} that $[x_1,x_n]_{\HAT{X}}$ 
fellow travels $\gamma$ in $\HAT{X}$ for longer and longer time as $n\map \infty$. 
By Lemma \ref{hor rays} there is a $K_{\ref{hor rays}}(\delta_0,k_0,\lambda_0)$-quasigeodesic ray, say $\beta$ such that $Hd(\gamma,\beta)\leq D_{\ref{hor rays}}(\delta_0,k_0,\lambda_0)$ in $\HAT{X}$. It follows that $x_n\map \beta(\infty)$ in $\HAT{X}$.
This shows that $\phi_X$ is well-defined and proves (1). Continuity of $\phi_X$ follows in the same way using
Lemma \ref{convg criteria 2} and Lemma \ref{hor rays}.

(2) If $\{x_n\}$ is a sequence in $X$ and $x_n\map \eta\in \partial \HAT{X}$ 
we can use Corollary \ref{electric cor2} to conclude that 
$\{x_n\}$ converges to a unique point of $\partial X$. Then another application of Lemma \ref{kapovich-rafi main cor}
would show that the limit is in $\partial_h X$. This gives the map $\psi_X$.
Then continuity of $\psi_X$ follows from Corollary \ref{electric cor2}
and the fact that $\psi_X=\phi^{-1}_X$ follow from the definitions of these maps. \qed

\smallskip
However to complete the discussion about the relation between the boundaries of $X$ and $\HAT{X}$ we 
shall need the following additional hypothesis on the collection $\{A_i\}$.

\begin{defn}
Suppose $Z$ is a metric space and $\{Z_i\}_{i\in I}$ is a collection of subsets of $Z$.
Then we call the collection $\{Z_i\}_{i\in I}$ {\em locally finite} if for all $z\in Z$ and $R>0$
$\{i\in I:B(z,R)\cap Z_i\neq \phi\}$ is a finite set.
\end{defn}

\begin{example}
Suppose $G$ is a finitely generated group and $X$ is the Cayley graph of $G$ with respect to a finite
generating set. If $H_1, H_2,\cdots, H_n$ are subgroups of $G$ then $\mathcal I=\{gH_i: g\in G, 1\leq i\leq n\}$
is a locally finite family of subsets of $X$ since distinct cosets of any same subgroup are disjoint 
and a finite radius ball of $X$ has only finitely many elements of $G$.
\end{example}
The next proposition is motivated by an analogous result proved in \cite[Lemma 6.12]{abbott-manning}.
This is complimentary to Lemma \ref{hor rays}.
\begin{prop}\label{electric main lemma}
Suppose the collection of quasiconvex subsets $\{A_i\}$ is locally finite.
Then for a quasigeodesic ray $\gamma$ of $X$, $\gamma(\infty)\in \Lambda(A_i)$ for some $i\in I$
if and only if $\gamma$ is a bounded set in $\HAT{X}$. 
\end{prop}

\proof  If $\gamma(\infty)\in \Lambda(A_i)$ for some $A_i$ then by Lemma \ref{limit set intersection}
$\gamma$ is contained in a finite neighborhood of $A_i$ in $X$. Thus $\gamma$ is a bounded
subset of $\HAT{X}$. 

Conversely suppose $\gamma$ is a bounded subset of $\HAT{X}$.
Let $D=Diam_{\HAT{X}}(\gamma)$. Let $\gamma(0)=x_0$ and let 
$\alpha_n$ be a geodesic in $\HAT{X}$ joining $x_0$ and $\gamma(n)$. 
Since $l(\alpha_n)\leq D$, each $\alpha_n$ passes through the cone points of at most $[D]$ 
distinct $k$-quasiconvex subsets $A_i$'s, where $[D]$ is the maximum integer smaller than or equal to $D$.
Hence, if necessary passing to a subsequence we may assume that all of the $\alpha_n$'s pass through the 
same number $m$ of cone points. {\em The proof is by induction on $m$.} 

Suppose for all $n\in \NN$, $A_1^n,A_2^n,...,A_m^n$ is the sequence of quasiconvex subsets from the collection $\{A_i\}$
whose cone points appear in this order on $\alpha_n$. Note that $d_X(x_0, A_1^n)\leq D$ for all $n$. 
Hence, by local finiteness of $\{A_i\}$, $\{A^n_1\}_{n\in \NN}$ is a finite set. Thus, up to passing to a 
further subsequence if necessary, we can assume that $A^n_1=A_1$ for all $n$. 
We note that for each $n$, $\alpha_n\cap X$ is the union of $m+1$ geodesic segments of $X$ sum
whose lengths is at most $D$. 

$m=1$: 
We note that $\alpha^{de}_n$ is the concatenation of two geodesic segments (of length at most $D$) in $X$ 
and a $\lambda_0$-quasigeodesic segment in $X$. By Lemma \ref{qc lemma 1}, the geodesic segments are $\delta_0$-quasiconvex and
the quasigeodesic segment is $D_{\ref{qc lemma 1}}(\delta_0, \lambda_0)$-quasiconvex.
Hence, by Lemma \ref{finite qc} $\alpha^{de}_n$ is 
$D_{\ref{finite qc}}(\delta_0, \delta_0+D_{\ref{qc lemma 1}}(\delta_0, \lambda_0),3)$-quasiconvex in $X$.
Let $K'=D_{\ref{finite qc}}(\delta_0, \delta_0+D_{\ref{qc lemma 1}}(\delta_0, \lambda_0),3)$.
Hence, $\gamma([0,n])$ is contained in a $(D+K')$-neighborhood of $A_1$ for all $n$.
It follows that $\gamma$ is contained in a $(D+K')$-neighborhood of $A_1$.
Therefore, $\gamma(\infty)\in \Lambda(A_1)$ by Lemma \ref{limit set intersection}.

$m>1$: We note that for each $n$, $\alpha_n\cap X$ is a disjoint union of $m+1$
geodesic segments of $X$. Now, the union of $\alpha_n\cap X$ and the sets $A_1=A^n_1,\cdots, A^n_m$ is a 
$K=D_{\ref{finite qc}}(\delta_0,k_0,m)$-quasiconvex subset of $X$ by Lemma  \ref{finite qc}. 
Let $S=\{t\in [0,\infty): \gamma(t)\in N_K(A'_1)\}$
where we let $A'_1$ to be the union of $A_1$ and the segment of $\alpha_n$ in $X$ from $x_0$ to $A_1$. There are two cases to
consider.

{\bf Case 1:} \underline{Suppose $S$ is unbounded.}
We note that if $\gamma(t_n)\in N_K(A'_1)$ for an unbounded sequence of numbers $t_n\in [0,\infty)$
then the nearest point projection of $\gamma$ on $A_1$ is of infinite diameter. Hence, by Lemma 
\ref{limit set intersection}, $\gamma(\infty)\in \Lambda(A_1)$ and we are done.

{\bf Case 2:} \underline{Suppose $S$ is bounded.} Let $t_1=\max S$.
 However, in this case, for all $n>t_1+1$ and $\gamma(t_1+1)$ is within the $(K+D)$-neighborhood of $A^n_{i_n}$ for some 
$i_n, 1<i_n\leq k$. Note that $d_X(x_0, A^n_{i_n})\leq t_1+1+K+D$. Using local finiteness
of the collection $\{A_i\}$, we see that the collection $\{A^n_{i_n}\}$ is finite. Hence, we may pass to a further 
subsequence and assume that $A^n_{i_n}$ is a fixed quasiconvex set, say $A_2\neq A_1$. 
we now replace each $\alpha_n$ by a new path
constructed by taking the concatenation of a geodesic $[x_0,x_2]_X$ in $X$ joining $x_0$ to $x_2\in A_2$ such that
$d_X(x_0,x_2)\leq t_1+1+K+D$, a segment joining joining $x_2$ to the cone point of $A_2$ followed by the segment
of $\alpha_n$ from the cone point of $A_2$ to $\gamma(n)$. Each of these paths have length at most $D+1+(t_1+1+K+D)$
and they go through the cone points of at most $k-1$ quasiconvex sets from the collection $\{A_i\}$.
Hence, we are done by induction on $m$. \qed

\smallskip

We note that the conclusion of Proposition \ref{electric main lemma} fails to hold if the collection $\{A_i\}$
is not locally finite as the following example shows.
\begin{example}
Let $X=[0,\infty)$ and $A_i=[0,i]$, $i\in \NN$. Then $X$ is hyperbolic and $A_i$'s are all convex subsets of $X$
of finite diameter. Hence, $\Lambda(A_i)=\emptyset$, $i\in \NN$. However, the geodesic 
$\gamma:[0,\infty)\map X$, $\gamma(t)=t$, is bounded in $\hat{X}$ since $\HAT{X}$ has diameter $2$. 
\end{example}

Proposition \ref{electric main lemma} and Lemma \ref{hor rays} immediately give the following:
\begin{theorem}\label{elec union thm}
If the collection of quasiconvex sets $\{A_i\}$ is a locally finite family in $X$ then
$\partial X=\partial_h X\cup \partial_v X$.
\end{theorem}
We note that a special case of this theorem was first proved by Abbott and Manning.
See \cite[Theorem 6.7, Remark 6.8, and Theorem 1.6]{abbott-manning}. 


\subsection{The main theorem}\label{section main thm}
In this subsection we shall prove the main technical theorem of our paper.
Following is our set up. Since the statement of the theorem is qualitative
rather than quantitative, we can make these assumptions for the sake of the
proof.
\begin{enumerate}
\item $Y\subset X$ are $\delta_0$-hyperbolic metric spaces where $Y$ has
the induced length metric from $X$. Let $\iota:Y \map X$ denote the inclusion map.
\item $\{B_i\}_{i\in I}$ is a locally finite collection of $k_0$-quasiconvex subsets in $X$. 
\item The inclusion $Y\map X$ is a $\rho_0$-proper embedding. 
\item $\{A_j\}$ is a collection of subsets of $Y$ such that each 
$A_j$ is contained in $B_i\cap Y$ for some $i$ and each $A_j$ is $k_0$-quasiconvex in
$X$ as well as in $Y$. 
\item Let $\HAT{X}$ be the space obtained from $X$ by coning the sets $B_i$ and let 
$\HAT{Y}$ be the space obtained from $Y$ by coning the sets $\{A_j\}$. Let
$\delta'_0=\delta_{\ref{mj-dahmani}}(\delta_0,k_0)$ Then both $\HAT{X}$ and $\HAT{Y}$ are $\delta'_0$-hyperbolic.
\end{enumerate}

We shall assume that $x_0\in Y$ is a fixed base point during the proof. We shall use the notation
$\phi_X$ and $\phi_Y$ from Theorem \ref{electric cor}.



\begin{theorem}\label{electric ct}
Suppose the inclusion $\HAT{Y}\map \HAT{X}$ satisfies Mitra's criterion. 
Then the inclusion $Y\map X$ admits the CT map $\partial Y\map \partial X$. Moreover, the CT map
is injective if and only if the CT map $\partial \HAT{Y}\map \partial \HAT{X}$ is injective.
\end{theorem}
\proof Since the inclusion $\HAT{Y}\map \HAT{X}$ satisfies Mitra's criterion
we have a CT map, say $g:\partial \HAT{Y}\map \partial \HAT{X}$. Now we consider the following map $h:\partial Y\map \partial X$:

\begin{equation*}
h(\xi)=
    \begin{cases}
        \xi & \text{if } \xi \in \partial_v Y,\\
        \phi^{-1}_X\circ g\circ\phi_Y(\xi) & \text{if } \xi \in \partial_h Y
    \end{cases}
\end{equation*}
We shall show that $h$ is the CT map $\partial Y\map \partial X$
by verifying the hypotheses of Corollary \ref{new CT lemma}. 
Suppose $\{x_n\}$ is a sequence in $Y$ and $x_n\map \xi\in \partial Y$.  
We need to verify that $x_n\map h(\xi)$ in $\bar{X}$.
The proof is divided into two cases.

{\bf Case 1.} \underline{Suppose $\xi\in \partial_h Y$.}
Then $x_n\map \phi_Y(\xi)$ in $\HAT{Y}$ by Theorem \ref{electric cor}. However the inclusion 
$\HAT{Y}\map \HAT{X}$ admits a CT map by hypothesis. Hence, $x_n\map g\circ \phi_Y(\xi)$ in $\HAT{X}$. 
Hence, again by Theorem \ref{electric cor} we have $x_n\map \phi^{-1}_X\circ g\circ\phi_Y(\xi)=h(\xi)$ in $\bar{X}$.

{\bf Case 2.} \underline{Suppose $\xi\in \Lambda(A_i)$ for some $i$.} 
Let $x\in A_i$. By Lemma \ref{ray in qc} there is a $K_{\ref{ray in qc}}(\delta_0,k_0)$-quasigeodesic ray $\gamma$ of $X$ contained in $A_i$ 
which joins $x$ to $\xi$. Then $\gamma$ is a $C_{\ref{quasigeodesic}}(\rho_0,K_{\ref{ray in qc}}(\delta_0,k_0),K_{\ref{ray in qc}}(\delta_0,k_0))$-quasigeodesic in $Y$ as well by Lemma \ref{quasigeodesic}.
There are two subcases to consider:

{\bf Subcase 1.} \underline{Suppose that $\{x_n\}$ is bounded in $\HAT{Y}$.}
Let  $l=\sup\{d_{\HAT{Y}}(x,x_n):n\in \NN\}$. Let $\alpha_n$ be a geodesic in $\HAT{Y}$ joining $x$ to $x_n$
and let $\gamma_n$ be a de-electrification of $\alpha_n$. By Lemma \ref{electric qc}
 $\gamma_n$ is a $K_{\ref{electric qc}}(\delta_0,k_0,l)$-quasiconvex path in $Y$ as well as in $X$.
Let $\beta_n$ be the concatenation of $\gamma$ and $\gamma_n$ for all $n\in \NN$. Then
$\beta_n$ is a $k=D_{\ref{finite qc}}(\delta_0,k_0,2)$-quasiconvex path in both $X$ and $Y$. We note that $\xi\in \Lambda_X(\beta_n)$.
Hence, we can choose, by Lemma \ref{ray in qc}, a $K_{\ref{ray in qc}}(\delta_0,k)$-quasigeodesic of $X$, say $\beta'_n\subset \beta_n$, 
joining $x_n$ to $\eta$. Then by Lemma \ref{quasigeodesic} it is a $C_{\ref{quasigeodesic}}(\rho_0,K_{\ref{ray in qc}}(\delta_0,k),K_{\ref{ray in qc}}(\delta_0,k))$-quasigeodesic in $Y$ as well.
Since $x_n\map \xi$ in $\bar{Y}$ we have $d_Y(x_0, \beta'_n)\map \infty$ by Lemma \ref{convg criteria}(2).
Since $Y$ is properly embedded in $X$, $d_X(x_0, \beta'_n)\map \infty$. That in turn implies that 
$x_n\map \xi$ in $\bar{X}$ again by Lemma \ref{convg criteria}(2).

{\bf Subcase 2.} \underline{Suppose that $\{x_n\}$ is unbounded in $\HAT{Y}$.} 
Passing to a subsequence, if needed, we may assume that $d_{\HAT{Y}}(x_0, x_n)>n$. 
Now, for all $R\in \NN$ and $n\geq R$, let $x^R_n\in Y$ be the farthest point of $[x_0,x_n]_Y$
such that $d_{\HAT{Y}}(x_0, x^R_n)=R$. We note that the sequence of geodesics $[x_0,x_n]_Y$ fellow 
travel $\gamma$ for longer and longer time as $n\map \infty$ by Lemma \ref{convg criteria 2}. 
This implies that 
$d_Y(x_0,[x_n^R,x_n]_Y)\map\infty$ as $n\map\infty$ for all large $R$ since the inclusion map $Y\map \HAT{Y}$
is Lipschitz. Hence, $x^R_n\map \xi$ in $\bar{Y}$
for large enough $R$. By the Subcase 1, for any such $R$ we have $x^R_n\map \xi$ in $\bar{X}$ too.

By choice of the points $x^R_n$ we see that $d_{\HAT{Y}}(x_0, [x^R_n, x_n]_Y)\geq R$.
Hence, by Corollary \ref{kapovich-rafi main cor} $d_{\HAT{Y}}(x_0, [x^R_n, x_n]_{\HAT{Y}})\geq R_1$
where $|R_1-R|$ is uniformly small.
Since the inclusion $\HAT{Y}\map \HAT{X}$ satisfies Mitra's criterion we have $R_2\geq 0$ depending on $R_1$ such that
$d_{\HAT{X}}(x_0, [x^R_n, x_n]_{\HAT{X}})\geq R_2$. Hence, by Corollary \ref{electric cor2} $d_X(x_0, [x^R_n, x_n]_X)\geq R_3$.
We note that $R_3\map \infty$ as $R\map \infty$. Now since $x^R_n\map \xi$ for all $R$ large enough, 
by Lemma \ref{diagonal sequence}  one may find an unbounded sequence of integers $\{m_R\}$ such that
$x^R_{m_R}\map \xi$.
On the other hand, this means
$d_X(x_0, [x^{R}_{m_R}, x_{m_R}]_X)\map \infty$. It follows that $x_{m_R}\map \xi$ in $\bar{X}$ as $R\map \infty$
by Lemma \ref{convg criteria}. Hence, by invoking Corollary \ref{new CT lemma} we are done. 

The last part of the theorem is clear from the definition of the map $h$. \qed 

\begin{cor}
Suppose we have the hypotheses (1)-(5) of the Theorem \ref{electric ct}. If $\HAT{Y}$ is a proper metric space, the inclusion
$\HAT{Y}\map \HAT{X}$ is a proper embedding and the CT map exists for the inclusion $\HAT{Y}\map \HAT{X}$ then 
the CT map exists for the inclusion $Y\map X$.
\end{cor}
\proof This is immediate from Remark \ref{CT iff Mitra} and Theorem \ref{electric ct}. \qed

\smallskip
Here is an example to show that the mere existence of the CT map for the inclusion $\HAT{Y}\map \HAT{X}$ is not
enough to guarantee the existence of the CT map for $Y\map X$. 
\begin{example}\label{counter example}
Suppose $X$ is obtained from the hyperbolic plane by gluing copies of $[0,\infty)$ at a sequence of points
of the form $P_n=(a_n,b_n)\in \HH^2$, where $a_n=1$ or $-1$ according as $n$ is odd or even. In other words, 
$X=(\HH^2\sqcup \NN \times [0,\infty))/\sim$ where $P_n$ is identified with $(n,0)\in \NN\times [0,\infty)$
and then one takes the natural length metric on this quotient.
Let $T_n$ denote the copy of $[0,\infty)$ glued to $P_n$.
Clearly $X$ is a hyperbolic metric space. Let $Y$ be the subspace of $X$ which is the union
of the part of the $y$-axis in $\HH^2$, the $T_n$'s and the horizontal (Euclidean) line segments joining 
$P_n$ to $(0,b_n)$. Then it is easy to see that for suitable
choices of $b_n$'s, $Y$ is properly embedded in $X$. The CT map for the inclusion $Y\map X$ does not exist because
both the sequences $\{(0,b_n)\}$ and $\{P_n\}$ converge to the same point at infinity for $Y$ but they converge to different
points for $X$. Let $A$ be the part of the $y$-axis in $\HH^2$. Clealrly $A$ is quasiconvex in both $X$ and $Y$.
Let $\HAT{X}, \HAT{Y}$ be the spaces obtained from $X,Y$ respectively by coning off $A$. 
Then it is clear that $\HAT{Y}$ is properly embedded in $\HAT{X}$ and that the CT map exists for the inclusion 
$\HAT{Y}\map \HAT{X}$ but Mitra's criterion fails to hold for $\HAT{Y}\map \HAT{X}$.
\end{example}

\begin{prop}[{\bf Converse to Theorem \ref{electric ct}}]\label{main converse}
Suppose we have the hypothesis (1)-(5) of Theorem \ref{electric ct} and the
 CT map $\partial \iota: \partial Y\map \partial X$ exists. 
Then the CT map $\partial \HAT{\iota}:\partial \HAT{Y}\map \partial \HAT{X}$ exists if and only if for any $B_i$, and
$\xi\in \Lambda_X(B_i)\cap Im(\partial \iota)$ we have $\xi \in \Lambda_X(A_j)$ for some $A_j$.
\end{prop}
\proof Suppose the CT map $\partial \HAT{\iota}:\partial \HAT{Y}\map \partial \HAT{X}$ exists. Clearly
we have the commutative diagram (Figure 1) below where the maps $\psi_Y, \psi_X$ are defined as in the proof of 
Theorem \ref{electric cor} whose images are $\partial_h Y$ and $\partial_h X$ respectively.
\begin{figure}[h]
	\[ \begin{tikzcd}[column sep=4.2em,row sep=3.2em]
		\partial\HAT{Y} \arrow{r}{\partial \HAT{\iota}} \arrow[swap]{d}{\psi_Y} & \partial\HAT{X} \arrow{d}{\psi_X} \\%
		\partial Y\arrow{r}{\partial \iota} & \partial X
	\end{tikzcd}
	\]
	\caption{}
	\label{pic 3}
\end{figure}

Suppose
$\eta\in \partial \HAT{Y}$ 
and $\{y_n\}$ is a sequence in $Y$ converging to  $\eta$ in $\HAT{Y}$.
Then $y_n\map \partial \HAT{\iota}(\eta) \in\HAT{X}$ since the CT map 
$\partial \HAT{\iota}:\partial \HAT{Y}\map \partial \HAT{X}$ exists. 
That in turn implies, again by Theorem \ref{electric cor}(2), that 
$y_n\map \psi_X(\partial \HAT{\iota}(\eta))\in \partial_h X$. 
Thus using the commutativity of the diagram, we have $\partial \iota(\partial_h Y)\subset \partial_h X$.

On the other hand, clearly $\partial \iota$ restricted to $\partial_v Y$ is injective and its image
is contained in $\partial_v X$.
$\partial Y\setminus \partial \HAT{Y}$ is clearly injective. Thus for any $B_i$, and $\xi\in \Lambda_X(B_i)\cap Im(\partial \iota)$,
implies $\xi\in \partial i(\partial_v Y)$, i.e. $\xi\in \Lambda_X(A_j)$ for some $A_j$.
 The converse is also similar and hence we skip the proof. \qed

\section{Consequences and examples}\label{section 4}
In this section we discuss several applications of the results of the previous section. We start with the following.

\subsection{Application to (relatively) hyperbolic groups}
Theorem \ref{electric ct} has the following immediate group theoretic consequences. 
Before stating the results we note that
in both the theorems mentioned below we have the following set-up. 

\begin{enumerate}
\item We have a finitely generated group $G$ and a finitely generated subgroup $H$ of $G$. We choose a finite generating set $S_G$ of $G$ which contains finite generating set $S_H$ of $H$ so that Cayley graph of $H$ with respect to $S_H$ naturally contained in the Cayley graph of $G$ with respect to $S_G$. For notational convenience, we denote these Cayley graphs of $H$ and $G$ by $H$ and $G$ respectively.

\item Moreover, there are subgroups $K_1, K_2,\cdots, K_n$ in $G$ and $K'_1,K'_2,\cdots, K'_m$ in $H$
such that the following hold:

(1*) {\em For all $1\leq i\leq m$ there is $g_i\in G$ and $1\leq r_i\leq n$ such that $K'_i= H\cap g_iK_{r_i}g^{-1}_i$.}

(2*)  {\em For all $g\in G$ and $1 \leq i\leq n$, there is $1\leq l\leq m$ and $h\in H$ such that 
$H\cap gK_i g^{-1}= hK'_lh^{-1}$.}

We note that for all $i, 1\leq i\leq m$ and all $h\in H$ we have $hK'_i\subset hg_iK_{r_i}g^{-1}_i$ by (1*). Thus if
$D=\max\{d(1,g_i): 1\leq i\leq m\}$ then $hK'_i\subset N_D(hg_i K_{r_i})$ for all $1\leq i\leq m$ and
$h\in H$ where the neighborhood is taken in a Cayley graph of $G$. 

\item {\em Let $\HAT{H}$ be the coned-off space obtained from $H$ by coning the various cosets of $K'_i$'s 
and $\HAT{G}$ be the coned-off space obtained from $G$ by coning off the cosets of the various $K_i$'s.} 

Let $\HAT{\HAT{G}}$ be the coned-off spaces obtained from $G$ by coning off the $D$-neighborhoods of the cosets of 
$K_i$'s in $G$. Then there are natural inclusion maps $\phi_1:\HAT{H}\map \HAT{\HAT{G}}$ and $\phi_2:\HAT{G}\map \HAT{\HAT{G}}$ respectively where the latter is a quasiisometry by Lemma \ref{basic cone-off 3}. 

\item {\em Let $\iota:H\map G$ denote the inclusion map and let  $\HAT{\iota}:\HAT{H}\map \HAT{G}$ be the natural map
such that $\phi_2\circ \HAT{\iota}=\phi_1$. }

We note that since $\phi_2$ is a quasiisometry, $\phi_1$ satisfies Mitra's criterion or admits the CT map- 
in case the coned-off spaces are hyperbolic, if and only if the same is true about $\HAT{\iota}$.
 
\end{enumerate}

\begin{theorem}\label{application 1}
Suppose $H< G$ are hyperbolic groups, $K_1, K_2, \cdots, K_n$ are quasiconvex subgroups of $G$
and $K'_1, K'_2,\ldots, K'_m$ are subgroups of $H$ which are all quasiconvex in $G$ such that (1*) and (2*) hold.
Then assuming the notation in (3) and (4) we have the following:

(1) If the map $\HAT{\iota}: \HAT{H}\map \HAT{G}$ satisfies Mitra's criterion then there is a CT map 
$\partial \iota: \partial H\map \partial G$. Moreover, if the CT map $\partial \HAT{\iota}$ is injective then $H$ is quasiconvex in $G$.

(2) Conversely, suppose the CT map for the inclusion $\iota: H\map G$ exists. Then the CT map for 
$\HAT{\iota}:\HAT{H}\map \HAT{G}$ exists if and only if for any coset $gK_i, 1\leq i\leq n$ and 
$\xi\in \Lambda_G(gK_i)$ either $\xi\not\in \partial i(\partial H)$
or there is $h\in H, 1\leq j\leq m$ such that $\xi\in \Lambda_G(h K'_j)$.
\end{theorem} 
\proof The first part of (1) is immediate from Theorem \ref{electric ct}. One notes that $K'_i$'s are quasiconvex
in $H$ as well by \cite[Lemma 2.2]{ilyakapovichcomb}. Also when the CT map $\partial H\map \partial G$ exists, 
it is injective if and only if so is the CT map
$\partial \HAT{H}\map \partial \HAT{G}$. Hence, the second part of (1) follows by Lemma \ref{inj of ct}.
(2) is immediate from Proposition \ref{main converse}. \qed

For background on relatively hyperbolic groups, one is referred to some standard references like \cite{bowditch-relhyp},\cite{farb-relhyp}
\cite{groves-manning}, and \cite{hru-rel}. One is referred to \cite[Definition 6.2]{hru-rel} for the definition of
a relatively quasiconvex subgroup of a relatively hyperbolic group. Now as another application of Theorem \ref{electric ct} 
one has the following.

\begin{theorem}\label{application 2}
Suppose $H<G$ are finitely generated groups and $K_i<G$, $1\leq i\leq n$ and $K'_j<H$, $1\leq j\leq m$
such that that (1*), (2*) hold. Moreover suppose that $G$ is hyperbolic relative to $K_i$'s and $H$
is hyperbolic relative to $K'_j$'s. Then with the notation in (3) and (4) we have the following:

If $\HAT{\iota}:\HAT{H}\map \HAT{G}$ satisfies Mitra's criterion then in level of the Bowditch 
boundaries of the groups the CT map $\partial_{rel}H\map \partial_{rel}G$ exists. Moreover, if the CT map is injective then 
$H$ is relatively quasiconvex in $G$.
\end{theorem}
{\bf Discussion.} Before we present a proof of the theorem let us first try to make sense of
the statement of the theorem.
\begin{enumerate}
\item By attaching hyperbolic cusps to a Cayley graph of $G$ along the various cosets
of $K_i$'s we get a hyperbolic metric space, say $X'$; e.g. see \cite[Definition 3.12]{groves-manning}. 
The {\em Bowditch boundary} (\cite{bowditch-relhyp}) of $G$ (with respect $\{K_i\}$) is defined 
to be $\partial_s X'$. It is a standard fact that $\partial_s X'$ independent on the choice of
Cayley graph of $G$. We shall denote it by $\partial_{rel}G$. Since $X'$ is a proper hyperbolic metric space,
$\partial_{rel}G$ is a compact metrizable space.
 
\item Let $D\geq 0$ be such that $K'_i\subset N_D(g_iK_{r_i})\cap H$ for all $1\leq i\leq m$.
After attaching hyperbolic cusps along the $D$-neighborhoods of all the cosets of the subgroups 
$\{K_i\}$ in a Cayley graph of $G$ we get a space, say $X$. Now, $X$ is also hyperbolic since 
$X$ is quasiisometric to $X'$ (see \cite[Theorem 1.2]{hruska-healy}).

\item Let $Y$ be the space obtained by attaching hyperbolic cusps to a Cayley graph of $H$ along 
the cosets of the subgroups in $\{K'_i\}$. The resulting space, say $Y$, is also hyperbolic.
Note that $\partial_{rel}H=\partial_s Y$ and $Y\subset X$. 

\item When the CT map $\partial Y\map \partial X$ exists then we say that the inclusion $H\map G$ 
admits the CT map in the level of the Bowditch boundaries. Lastly it is clear that $H$ naturally 
acts on $Y$ isometrically and hence it induces an action on $\partial_{rel}Y$. For a proof see 
\cite[Section 3]{groves-manning}. This action is geometrically finite by \cite[Theorem 5.4]{hru-rel}.
\end{enumerate}

{\em Proof of Theorem \ref{application 2}:} Suppose $\HAT{X}$ and $\HAT{Y}$ are the spaces
obtained from $X,Y$ respectively by coning off all the hyperbolic cusps in $X$ and $Y$ respectively. 
We note that $\HAT{G}$ is naturally quasiisometric to $\HAT{X}$. 
Thus the inclusion map $\HAT{Y}\map \HAT{X}$ satisfies Mitra's criterion. It is a standard fact 
that the hyperbolic cusps are uniformly quasiconvex in the respective cusped spaces 
(see \cite[Lemma 9.2(1)]{kapovich-sardar}).
Hence, all the hypotheses of Theorem \ref{electric ct} checks out. It follows that
CT map exists for the inclusion $Y\map X$.

For the second part we note that $\partial_{rel}H$ is equivariantly homeomorphic to the image of the CT map,
since the Bowditch boundaries are compact, Hausdorff spaces as mentioned in the discussion above. 
Hence, $H$-action on the image of the CT map is geometrically finite. We note that the $H$-action on
$\partial_{rel}H$ is minimal (see \cite[p. 4]{bowditch-relhyp}) and hence the $H$-action on the image of the CT 
map is also minimal. Since the limit set of $H$ in $\partial_{rel}G$ is the unique $H$-invariant closed 
set (see e.g. \cite{coornaert}) on which the $H$-action is minimal, the limit set of $H$ in $\partial_{rel}G$ 
is precisely the image of the CT map. Thus the action of $H$ on its limit set in $\partial_{rel}G$ is 
geometrically finite, whence $H$ is relatively quasiconvex (see \cite[Definition 6.2]{hru-rel}). \qed

\smallskip
We note that a statement similar to the second part of Theorem \ref{application 1} can be formulated
in this case too. But it will require more facts about relatively hyperbolic groups to be recalled and
hence we skip it.

\subsection{Applications to complexes of groups}\label{4}
Complexes of groups are natural generalizations of graphs of groups (see \cite{serre-trees}).
Gerstein and Stallings in \cite{gerstein-stalling} gave the first instances of complexes of groups. 
Subsequently, a more general theory of complexes of groups was studied independently by Corson \cite{corson1} 
and Haefliger \cite{haefliger}. Here, we briefly recall basic facts about some rather special types of complexes 
of groups needed for our purpose. For more details, one is referred to \cite{bridson-haefliger},\cite{haefliger}.

For the rest of this subsection, we shall denote by $\YY$ a finite connected simplicial complex. Let $\BB(\YY)$
denote the directed graph whose vertex set is the set of simplices of $\YY$ and given two simplices $\tau\subset \sigma$
we have a directed edge $e$ from $\sigma$ to $\tau$. In this case we write $e=(\sigma, \tau)$, $o(e)=\sigma$ and $t(e)=\tau$.
Two directed edges $e,e'$ are said to be {\em composable} if $t(e)=o(e')$. In that case the composition
is denoted by $e*e'$.

\begin{defn}[{\bf Complex of groups}] 
A complex of groups $(\mathcal{G},\YY)=(G_{\sigma},\psi_a,g_{a,b})$ over $Y$ consists of the following data:
	\begin{enumerate}
		\item For each $\sigma\in V(\BB(\YY))$, there is a group $G_{\sigma}$- called the local group at $\sigma$.
		
		\item For each edge $e\in E(\BB(\YY))$, there is an injective homomorphism $\psi_e:G_{i(e)}\rightarrow G_{t(e)}$.
These homomorphisms are referred to as the local maps.
		
		\item For each pair of composable edges $e,e'\in E(\BB(\YY))$, a twisting element $g_{e,e'}\in G_{t(e)}$ with the following properties:
		
		(i) $Ad(g_{e,e'})\psi_{e*e'}=\psi_e\psi_{e'}$ where $Ad(g_{e,e'})$ denotes the conjugation by $g_{e,e'}$,
		
		(ii)(Cocycle condition) $\psi_e(g_{e',e''})g_{e,e'*e''} = g_{e,e'}g_{e*e',e''}$ for each triple 
$e,e',e''$ of composable edges of $E(\BB(\YY))$. 
	\end{enumerate} 
\end{defn}
The cocycle condition is empty if dimension of $\YY$ is $2$. If $\YY$ is $1$-dimensional then the complex of groups over $\YY$ 
is the same as the graph of groups over $\YY$. However, given a complex of groups one can define
a complex of space.

\begin{defn} [\bf{Complexes of spaces}]\cite[Definition 1.3]{martin1}
A complex of spaces $C(\YY)$ over $\YY$ consists of the following data:
	\begin{enumerate}
		\item For every simplex $\sigma$ of $\YY$, a CW-complex $C_{\sigma}$.
		\item For every pair of simplices $\sigma\subset \sigma'$, an embedding $\phi_{\sigma,\sigma'}:C_{\sigma'}\rightarrow C_{\sigma}$ called a gluing map such that for every $\sigma\subset\sigma'\subset\sigma''$, we have $\phi_{\sigma,\sigma''}=\phi_{\sigma,\sigma'}\phi_{\sigma',\sigma''}$.
	\end{enumerate}
\end{defn}

The {\em topological realization} $|C(\YY)|$ of the complex of spaces $C(\YY)$ is the following quotient space:
\begin{center}
	$|C(\YY)|=(\bigsqcup\limits_{\sigma\subset \YY}\sigma\times C_{\sigma})/_\sim$
\end{center}
Here all $\sigma$'s are given the subspace topology from the Euclidean spaces and 
$(i_{\sigma,\sigma'}(x),s)\\
\sim(x,\phi_{\sigma,\sigma'(s)})$ for $x\in \sigma\subset\sigma'$, 
$s\in C_{\sigma'}$ and $i_{\sigma,\sigma'}:\sigma\hookrightarrow\sigma'$ is the  inclusion.

Given a complex of groups $(\GG, \YY)$ over $\YY$, for each simplex $\sigma$ of $\YY$ one takes any simplicial complex (generally
a $K(G_{\sigma},1)$-space), say $\YY_{\sigma}$, with a base point such that (1) $\pi_1(\YY_{\sigma})\simeq G_{\sigma}$ and (2)
for every pair of simplices $\tau\subset \sigma$ one has a base point preserving continuous map $\YY_{\sigma}\map \YY_{\tau}$ 
which induces the homomorphism $\psi_{(\sigma, \tau)}$ at the level of fundamental groups. 
This defines a complex of spaces over $\YY$. By abuse of terminology we also call the realization, say
$\YB$, of this to be a complex of spaces as well. Note that we have a natural simplicial map $\YB\map \YY$.

The {\bf fundamental group} $\pi_1(\GG,\YY)$ of $(\GG, \YY)$ is defined to be $\pi_1(\YB)$. 
It is a standard consequence of van Kampen theorem that this is independent of the complex of spaces $C(\YY)$ thus chosen.
For a proof of this one is refered to \cite[Remarks, p. 88]{corson}. When the homomorphisms $\pi_1(\YY_{\sigma})\map \pi_1(\YB)$
induced by the inclusion map $\YY_{\sigma}\map \YB$ are injective for all simplices $\sigma$ of $\YY$ then we say 
that the complex of groups $(\GG, \YY)$ is {\bf developable}. (One can also define developable complexes of group 
as in \cite{haefliger-cplx},\cite{bridson-haefliger}. However, these definitions are equivalent by the results 
of \cite[Section 2]{corson}.)

{\bf Note:} {\em For the rest of the paper we shall always assume that all our complexes of groups are
developable. We shall denote $\pi_1(\GG, \YY)$ by $G$ for the rest of this section.}

\medskip
{\bf Development or universal cover of a developable complex of groups}\\
Suppose $\tilde{\YB}\map \YB$ is a universal cover of $\YB$. Then as in \cite{corson} one considers the composition
$\tilde{\YB}\map \YB\map \YY$, say $f$, and collapses the connected components of $f^{-1}(y)$ for all $y\in \YY$.
The resulting simplicial complex, say $B$, is called the {\em development} or the {\em universal cover} of 
$(\GG, \YY)$.  Note that $\tilde{\YB}$ has an induced simplicial complex structure from $\YB$ so that the map 
$\tilde{\YB}\map \YB$ is simplicial. Then it follows that $B$ is also a simplicial complex and the natural map 
$B\map \YB$ is simplicial. Moreover, one may show that the development is independent of the complex of spaces chosen for the complex
of groups; see \cite[Section 2]{corson} for a proof. It is in fact the following (see \cite{martin-univ}) simplicial complex:

	\begin{center}
		$B:=(G\times (\bigsqcup\limits_{\sigma\subset Y}\sigma))/\sim$
	\end{center}
	where $(gg',x)\sim (g,x)$ for $g\in G,g'\in G_{\sigma}$, $x\in \sigma$ and 
$(g,i_{\sigma',\sigma}(x))\sim (ge,x)$ for $g\in G,x\in \sigma$ where $i_{\sigma',\sigma}:\sigma'\hookrightarrow \sigma$ is the inclusion map and $e=(\sigma,\sigma')$.

From either of the descriptions of the development it follows that $G$ has a natural simplicial action on $B$ with quotient $\YB$.

{\bf Notation.} In the remaining part of the paper, $X$ will denote $\tilde{\YB}$ and  $p:X\map B$ will denote the
natural quotient map. 

\begin{rem}\label{remk cplx}
(1) However, we do not work with the CW topology on $B$. We rather view $B$ as the quotient metric space (see \cite[Chapters I.5, I.7]{bridson-haefliger}
obtained by gluing standard Euclidean simplices. Then clearly the $G$-action on $B$ is through isometries.

(2) When $\mathbb Y$ is a finite CW or simplicial complex then we can put a length metric on $\mathbb Y$ 
\cite[Chaper I.7]{bridson-haefliger} which then naturally gives rise to a length metric on $X$ 
\cite[Definition 3.24, Chapter I.3]{bridson-haefliger} and $X$ becomes a geodesic metric space
and the $G$-action on $B$ is (properly discontinuous, cocompact and) through isometries. Also, in this case the map
$p:X\map B$ is $1$-Lipschitz and $G$-equivariant.

(3) In case all the face groups are hyperbolic, one can choose a complex of spaces $C(\YY)$ where the spaces for each
face is a finite CW or simplicial complex since hyperbolic groups are finitely presented. Thus $|C(\YY)|$ a finite
CW (or simplicial) complex.
\end{rem}

The following proposition follows from \cite[Theorem 5.1]{charney-crisp}. For an alternative treatment one may look
up \cite[Section 3]{mbdl2}. 
\begin{prop}\label{development is cone off}
Let $\YY$ be a finite simplicial complex and let $(\GG,\YY)$ be a developable complex of groups. Let $p: X\map B$ be
as in (2) of the Remark \ref{remk cplx}.
Suppose $x\in X$. Consider the orbit map $\mathcal O: G\map X$ given by $g\mapsto g.x$.
There is a constant $D\geq 0$ such that the following holds: \\Let $g\in G$, $\sigma\subset \YY$ be a face and 
$y\in \sigma$ is the barycenter of $\sigma$.

Then (1) $Hd(\mathcal O(gG_{\sigma}), p^{-1}([g,y]))\leq D$ where $[g,y]$ denotes the equivalence
class of $(g,y)$ in $B$ as defined above.

(2) The spaces obtained from $X$ by coning off the following two collections of subsets are naturally quasiisometric:

\noindent
(i) Inverse images of all points in $B$ under $p$.\\
(ii) Inverse images under $p$ of only the barycenters of the various simplices of $B$.

(3) The map $p\circ \mathcal O$ induces a $G$-equivariant quasiisometry $\HAT{G}\map B$ where $\HAT{G}$ is the coned off
Cayley graph of $G$ obtained by coning the various cosets of the face groups of $(\GG, \YY)$.
\end{prop}



Following is the set up for the main theorem of this section.

\begin{itemize}
\item (H1) Suppose $(\mathcal{G},\YY)$ is a developable complex of groups 
with the development $B$.
\item (H2) Suppose $G$ is hyperbolic and all the face groups are quasiconvex in $G$. 
\item (H3) Suppose $\YY_1$ is a connected subcomplex of $\YY$ and $(\mathcal{G},\YY_1)$ is the subcomplex of groups 
obtained by restricting $(\mathcal{G},\YY)$ to $\YY_1$. Then by \cite[Corollary 2.15, Chapter III.C]{bridson-haefliger}
$(\mathcal{G},\YY_1)$ is a developable complex of
groups. Let $G_1=\pi_1(\GG, \YY_1)$ and let $B_1$ be the development. 
\item (H4) Suppose the natural homomorphism $G_1\map G$ is injective. Then we have a natural map $B_1\map B$.
\item (H5) Suppose $G_1$ is also hyperbolic. 
\end{itemize}
Then we have the following.

\begin{theorem}\label{acyl ct thm}
(1) The spaces $B, B_1$ are hyperbolic.

(2) If the natural map $B_1\rightarrow B$ satisfies Mitra's criterion
then there exists a Cannon-Thurston map for the inclusion $G_1\rightarrow G$. 
Moreover, $G_1$ is quasiconvex in $G$ if and only if the Cannon-Thurston map for 
$B_1\rightarrow B$ is injective.
\end{theorem}
\proof Let $C(\YY)$ be a complex of spaces for the complex of groups $(\mathcal{G},\YY)$
where the space for each simplex is a finite simplicial complex. Let $C(\YY_1)$
be the restriction of $C(\YY)$ to $\YY_1$. Let $\mathbb Y$ be the topological realization
of $C(\YY)$ and let $\mathbb Y_1$ be the topological realization of $C(\YY_1)$.
Then we have the following commutative diagram (figure \ref{pic 1}) where the horizontal
maps are inclusions.

\begin{figure}[h]
	\[ \begin{tikzcd}[column sep=4.2em,row sep=3.2em]
		\mathbb{Y}_1 \arrow{r} \arrow[swap]{d} & \mathbb{Y} \arrow{d} \\%
		\YY_1\arrow{r} & \YY
	\end{tikzcd}
	\]
	\caption{}
	\label{pic 1}
\end{figure}


Let $f:X\map \mathbb Y$ and $f_1:X_1\map \mathbb Y_1$ be universal covers.
Then we have a commutative diagram as below where the top horizontal map
can be assumed to be an inclusion map since $\pi_1(\YB_1)\map \pi_1(\YB)$ is injective.

\begin{figure}[h]
	\[ \begin{tikzcd}[column sep=4.2em,row sep=3.2em]
		X_1 \arrow{r} \arrow[swap]{d}{f_1} & X \arrow{d}{f} \\%
		\mathbb{Y}_1\arrow{r} & \mathbb{Y}
	\end{tikzcd}
	\]
	\caption{Horizontal maps are inclusion}
	\label{pic 2}
\end{figure}

Since the top horizontal map is $G_1$-equivariant by choosing $x\in X_1$ and using 
Proposition \ref{development is cone off} we have a commutative diagram
as below. 

\begin{figure}[h]
	\[ \begin{tikzcd}[column sep=4.2em,row sep=3.2em]
		\HAT{G_1} \arrow{r} \arrow[swap]{d} & \HAT{G} \arrow{d} \\%
		B_1\arrow{r} & B
	\end{tikzcd}
	\]
	\caption{Vertical maps are quasiisometries}
	\label{pic 3}
\end{figure}


Here the vertical maps are quasiisometries. Hence, (1) follows  
from Proposition \ref{mj-dahmani} and also $B_1\map B$ admits the CT map or satisfies Mitra's
criterion if and only if the same is true for the map $\HAT{G}_1\map \HAT{G}$. Moreover, any of these two maps is injective if and only if
so is the other one. Hence, the first part of (2) in the theorem follows from the first part of Theorem \ref{application 1}(1).
It also follows that in this case the CT map $\partial G_1\map \partial G$ is injective if and only if
so is the CT map $\partial B_1\map \partial B$. Thus the second part of (2) follows from the second part of
Theorem \ref{application 1}(1). \qed

\smallskip
Below we mention two instances of complexes of groups where the fundamental groups are
hyperbolic and all the face groups are quasiconvex.

\subsubsection{Complexes of groups with finite edge groups}
In this section we consider developable complexes of groups whose edge groups are finite and whose developments are hyperbolic
which may or may not be CAT(0).
The following proposition then is immediate from the work of \cite{bowditch-relhyp}, and \cite{pedroja-fine-graph}. We include the statement
in this paper because we could not find it in the literature.

\begin{theorem}\label{gen fin edge}
Suppose $(\GG,\YY)$ is a developable complex of groups such that the edge groups are finite and the universal cover of $(\GG,\YY)$ is a hyperbolic
metric space. Then the fundamental group of $(\GG,\YY)$, say $G$, is hyperbolic relative to the vertex groups $\{G_v: v\in V(\YY)\}$.
\end{theorem}
{\em A word about the proof:} Bowditch (\cite{bowditch-relhyp}) showed that a finitely generated group is hyperbolic relative 
to a finite set of finitely generated infinite subgroups $\{H_i\}$ if and only if there is a {\em fine} hyperbolic graph $X$ on which 
$G$ has a cofinite action and such that the edge stabilizers
are finite and the infinite vertex stabilizers are precisely the conjugates of $\{H_i\}$ in $G$. Using this criterion one 
only needs to check that the $1$-skeleton of the development of $(\GG, \YY)$ is a fine graph. 
This follows from \cite[Corollary 2.11, Theorem 1.3]{pedroja-fine-graph}. 

Now, in the above theorem if one assumes in addition that the vertex groups are hyperbolic then one has the following:

\begin{theorem}\label{combhyp finite edge}
Suppose $(\GG,\YY)$ is a developable complex of groups such that the vertex groups are hyperbolic and the edge groups are finite. 
Suppose the universal cover of $(\GG,\YY)$ is  hyperbolic. Then the fundamental group of $(\GG,\YY)$ is a hyperbolic group
and vertex groups are quasiconvex in $G$.
\end{theorem}
In fact, one may use Theorem \ref{gen fin edge} along with \cite[Corollary 2.41]{osin-book} or \cite[Theorem 2.4]{hamnestadtrelative}
for the proof of Theorem \ref{combhyp finite edge}.


Here is the set up for the main result of this subsection.
\begin{itemize}
\item Suppose $(\GG,\YY)$ is a complex of groups as in Theorem \ref{gen fin edge} with $\pi_1(\GG,\YY)=G$
and with universal cover $B$. 
\item Suppose $\YY_1\subset \YY$ is a connected subcomplex and $(\GG,\YY_1)$ is the restriction of $(\GG,\YY)$ to $\YY_1$ such that at the level 
of fundamental groups we have an injection. 
Let $H=\pi_1(\GG,\YY_1)$ and let $B_1$ be the universal cover of $(\GG,\YY_1)$. 
\item Assume that $B_1$ is hyperbolic.
We note that this implies $H$ is hyperbolic relative to the vertex groups in $(\GG,\YY_1)$ by
Theorem \ref{gen fin edge}.

\end{itemize}
Then, by Theorem \ref{application 2} we have the following:

\begin{theorem}
There is a Cannon-Thurston map $\partial_{rel} H\map \partial_{rel} G$ at the level of the Bowditch boundaries of $H$ and
$G$ if Mitra's criterion holds for the inclusion $B_1\map B$.

Moreover, $H$ is a relatively quasiconvex subgroup of $G$ if and only if the CT map for $B_1\map B$ is injective.
\end{theorem}
Similarly, by Theorem \ref{application 1} we immediately have the following:
\begin{theorem}
Suppose $(\GG,\YY)$ satisfies the hypotheses of Theorem \ref{combhyp finite edge}.  Then
	there is a Cannon-Thurston map $\partial H\map \partial G$ at the level of the Gromov boundaries of $H$ and
	$G$ if Mitra's criterion holds for the inclusion $B_1\map B$.
	
	Moreover, $H$ is quasiconvex in $G$ if and only if the CT map for $B_1\map B$ is injective.
\end{theorem}
\begin{comment}
\proof \RED{The proof needs more preparation: define rel qc, rel cayley graph, equiv of defns of rel qc}
The first part follows immediately from Theorem \ref{application 2} and Theorem \ref{combhyp finite edge}.
Suppose $H$ is a relatively quasiconvex subgroup of $G$. Then by \cite[Theorem 9.1]{hru-rel}, 
$H$ is a relatively hyperbolic group 
with respect to infinite vertex groups of $(\GG,Y_1)$. Again, by \cite[Theorem 10.1]{hru-rel}, the natural inclusion $H\rightarrow G$ 
induces a qi embedding between relative Cayley graphs of $H$ and $G$, respectively. But we know that relative Cayley graph is quasiisometric to coned-off Cayley graph (see \cite{osin-book}). Also, coned-off Cayley graphs of $G,H$ with respect to peripheral structure are quasiisometric to $B,B_1$, respectively. Hence, there exists a qi embedding from $B_1$ to $B$.
	
	Conversely, suppose that the inclusion $B_1^{(1)}\rightarrow B^{(1)}$ is a qi embedding. Above, we saw that the $1$-skeleton of $B$ is a fine hyperbolic graph. Now, $B^{(1)}$ is a qi embedded subgraph of $1$-skeleton of $B$ and also it is $H-$invariant. Hence by \cite[Definition 1.3]{pedraoja-wise}, we see that $H$ is a relatively quasiconvex subgroup of $G$. 
\qed
\end{comment}

\subsubsection{Acylindrical complexes of groups}

In \cite{martin1} (along with \cite[Corollary, p. 805]{martin-univ}) A. Martin proved the following theorem for complexes of 
hyperbolic groups. 

\smallskip
{\bf Theorem} \textup{(\cite[p. 34]{martin1}\label{martin})} {\em Let $(\mathcal{G},\YY)$ be a developable complex of groups such that
the following holds:
\begin{itemize}
\item (M1) $\YY$ is a finite connected simplicial complex,
\item (M2) all the face groups are hyperbolic and local maps are quasiisometric embeddings,
\item (M3) the universal cover of $(\mathcal{G},\YY)$ is a CAT(0) hyperbolic space, and
\item (M4) the action of $\pi_1(\mathcal{G},\YY)$- the fundamental group of $(\mathcal{G},\YY)$, on the
development is acylindrical.
\end{itemize} 
Then $\pi_1(\mathcal{G},\YY)$ is a hyperbolic group and the local groups are quasiconvex in $\pi_1(\mathcal{G},\YY)$.}

\smallskip
In what follows, the above theorem is referred to as Martin's theorem. Using this theorem we deduce the following corollary to Theorem \ref{acyl ct thm}.

\begin{cor}\label{martin application}
Let $(\mathcal{G},\YY)$ be a complex of groups satisfying the conditions (M1)-(M4) of the above theorem. 
Let $\YY_1$ be a connected subcomplex of $\YY$ and 
let $(\mathcal{G},\YY_1)$ be the subcomplex of groups obtained by restricting $(\mathcal{G},\YY)$ to $\YY_1$. Suppose 
the following conditions hold:

\begin{enumerate}
\item $(\mathcal{G},\YY_1)$ also satisfies (M1)-(M4),
\item the natural homomorphism $H=\pi_1(\mathcal{G},\YY_1) \map G=\pi_1(\mathcal{G},\YY)$ is injective,
\item  the natural map  $B_1\rightarrow B$ satisfies Mitra's criterion where $B_1,B$ are the universal covers
of $(\mathcal{G},\YY_1)$ and $(\mathcal{G},\YY)$ respectively. 
\end{enumerate}

Then there exists a Cannon-Thurston map for the inclusion $H\rightarrow G$. 
Moreover, $H$ is quasiconvex in $G$ if and only if the Cannon-Thurston map for $B_1\rightarrow B$ is injective.
\end{cor}

\section{Other applications and examples}\label{5}

In this section we prove a few other related results about complexes of groups. 

\subsection{Polygons of groups} 

Following is the main result of this subsection that we obtain as an application of Theorem \ref{acyl ct thm}.

\begin{theorem}\label{edge qc}
Suppose $\mathcal{Y}$ is a regular Euclidean polygon with at least $4$ edges and let $(\GG,\YY)$ be a simple complex of hyperbolic groups
such that the following are satisfied:
\begin{enumerate}
\item In any vertex group, intersection of two different edge groups is equal to the subgroup coming from the barycenter of $\YY$. 
\item The universal cover $B$ of $(\GG,\YY)$ is a hyperbolic metric space.
\item The action of $G=\pi_1(\GG,\YY)$ on $B$ is acylindrical. 
\end{enumerate}

Suppose $\YY_1$ is an edge of $\YY$ and $(\GG,\YY_1)$ is the restriction of $(\GG,\YY)$ to $\YY_1$. Let $H=\pi_1(\GG,\YY_1)$.
Then $G$ is a hyperbolic group and $H$ is a quasiconvex subgroup of $G$.
\end{theorem} 

\proof We begin with the following simple observation.
By \cite[12.29, p. 390, II.12]{bridson-haefliger}, $(\GG,\YY)$ is non-positively curved and 
hence it is developable. Also
$B$ is a piecewise Euclidean polygon complex, and by \cite[Theorem 7.50, I.7]{bridson-haefliger}
$B$ is a complete geodesic metric space since all simplices are isometric to each other.
It follows that $B$ is a CAT(0) space (see \cite[paragraph after Theorem 4.17, p. 562, III.C]{bridson-haefliger}). 
Thus all the hypotheses of Martin's theorem are satisfied for $(\GG,\YY)$ and hence $G$ is a hyperbolic group.

Next part of the proof makes use of the following three lemmas. We shall need the following definition. 
Let $d$ denote the metric on $B$ and let $d_P$ denote the metric on any polygon $P$ in $B$. We also assume that each edge in $B$ is of length $1$.

\begin{defn}{\em (\cite[Chapter I.7, Definition 7.8]{bridson-haefliger})}
	Suppose $x\in B$. Define $$\epsilon(x)=\inf\{\epsilon(x,P): \text{$P$ is a polygon in $B$ containing $x$} \}$$
	where $\epsilon(x,P):=\{d_P(x,e):e \text{ is an edge of $P$ not containing $x$}\}$.
\end{defn}
The lemma below can be proved in the same way as \cite[Lemma 7.9]{bridson-haefliger}. 
For the sake of completeness, we include a sketch of proof. 

\begin{lemma}\label{pologon lemma 1}
	Fix $x\in B$. If $y\in B$ is such that $d(x,y)< \epsilon(x)$ then any polygon $P$
	which contains $y$ also contains $x$ and $d(x,y)= d_P(x,y)$.
\end{lemma}
{\em Sketch of proof:} To prove the lemma, it is sufficient to show that if 
$\Sigma=(x=x_0,x_1,...,x_n=y)$ is an $m$-string (see \cite[I.7, Definition 7.4]{bridson-haefliger}) 
of length $l(\Sigma)<\epsilon(x)$, with $m\geq 2$, then $\Sigma'=(x_0,x_2,...,x_n)$ is an $(m-1)$-string 
with length $l(\Sigma')\leq l(\Sigma)$. Now, by definition of $m$-string, there is a polygon $P_1$ such 
that $x_1,x_2\in P_1$. Since $l(\Sigma)< \epsilon(x)$, $x_0\in P_1$. By triangle inequality, 
$d_{P_1}(x_0,x_2)\leq d_{P_1}(x_0,x_1)+d_{P_1}(x_1,x_2)$. Thus, $(x_0,x_2,...,x_n)$ is an 
$(m-1)$-string of length less that or equal to $l(\Sigma)$.  \qed

\smallskip
The following lemma must be well known or obvious but we could not find a written proof anywhere in the literature.
Hence we include a proof here.
\begin{lemma}\label{iso. embedd. poly.}
	The inclusion of any polygon $P$ in $B$ is an isometric embedding where $P$ is given its Euclidean metric.
\end{lemma}
\proof
Suppose $x,y\in P$ are two interior points and $[x,y]_P$ is the straightline joining $x,y$ in $P$.
For all $z\in [x,y]_P$ the ball of radius $\epsilon(z)$ in $P$ is isometrically embedded in $B$ by Lemma \ref{pologon lemma 1}.
Hence, a small neighborhood of $z$ in $[x,y]_P$ is isometrically embedded in $B$. This shows that $[x,y]_P$ is a local
geodesic in $B$. Since $B$ is a CAT(0) space, it follows that $[x,y]_P$ is a geodesic in $B$. Thus the interior of $P$
is isometrically embedded in $B$. 
The lemma follows easily from this and the following standard fact (see \cite[Proposition 1.4(1), II.1]{bridson-haefliger}):\\ 
 {\em Suppose $\{z_n\},\{z'_n\}$
are two sequences in a CAT(0) space $Z$ and $z,z'\in Z$ such that $z_n\map z, z'_n\map z'$. Let 
$\alpha_n:[0,1]\map Z$ be a constant speed geodesic joining  $z_n,z'_n$ in $Z$ for all $n\in \NN$. Then $\alpha_n$ converges 
uniformly to a constant speed geodesic $\alpha:[0,1]\map Z$ joining $z,z'$.} \qed
\begin{lemma}\label{polygon lemma 2}
	Let $P_1$ and $P_2$ be two distinct polygons in $B$. Suppose $e_1\subset P_1$ and $e_2\subset P_2$ are edges
such that $e_1\cap P_2=e_2\cap P_1=e_1\cap e_2$ is a vertex. Then the concatenation of $e_1$ and $e_2$ is a geodesic in $B$.
\end{lemma}
\begin{proof}
	 Let $\alpha$ be the concatenation of $e_1$ and $e_2$. By Lemma \ref{iso. embedd. poly.}, $e_1,e_2$ are geodesics in $B$. 
Let $v:= e_1\cap e_2$. To prove the lemma, it is sufficient to show that a small neighborhood of $v$ in $\alpha$ embeds 
isometrically in $B$. Let $I$ be the open ball of radius $1/4$ around $v$ in $\alpha$. Suppose $x,y$ are the endpoints 
of $I$ lying in $e_1,e_2$ respectively. Note that $\Sigma_0=(x,v,y)$ is a $2$-string in $B$ and $l(\Sigma_0)=1/2$. 
If there is no string in $B$, except $\Sigma_0$, connecting $x$ and $y$ then we are done. Therefore, we show that if 
$\Sigma=(x_0=x,x_1,...,x_n=y)$ is any other $n$-string in $B$ then $l(\Sigma)>l(\Sigma_0)$. By definition of an $n$-string, 
there exists a sequence of polygons, say $(P_0',P_1',...,P_{n-1}')$, such that $x_i,x_{i+1}\in P_i'$ for $i=0,1,...,(n-1)$.
 Without loss of generality, we can assume that $x_i,x_{i+1}$ lie on different sides of $P_i'$ for $0\leq i\leq (n-1)$. Note 
that if $x_1$ belongs to a side of $P_0'$ not containing $v$ then $d_{P_0'}(x,x_1)>1/2$ as the angle at each vertex 
of $P_0'$ is at least $\pi/2$ and hence $l(\Sigma)>1/2$. Thus we assume that $x_1$ belongs to a side of 
$P_0'$, say $e_3$, containing $v$. Since the angle between $e_1$ and $e_3$ is at least $\pi/2$, 
$d_{P_0'}(x,x_1)> 1/4$. By the same reasoning, $x_2$ belongs to a side of $P_1'$ containing $v$. By continuing in 
this way either $x_{n-1}$ belongs to a side of $P_{n-1}'$ containing $v$ or $x_{n-1}$ lies on a side of $P_{n-1}'$ not 
containing $v$. In either case $d_{P_{n-1}'}(x_{n-1},x_n)>1/4$. Hence, 
$l(\Sigma)\geq d_{P_0'}(x,x_1)+d_{P_{n-1}'}(x_{n-1},x_n)>1/2$. This completes the proof.
\end{proof}

{\em Proof of Theorem \ref{edge qc} continued:}
	Let $v,w$ be the vertices and $e$ the barycenter of $Y_1$. 
Then $H=G_v\ast_{G_e}G_w$. Let $B_1$ denote the Bass-Serre tree for the amalgam decomposition of $H$. 
We note that there is a natural homomorphism $H\map G$ and a natural simplicial map $B_1\map B$.

{\bf Claim:} The natural map $B_1\map B$ is an isometric embedding.\\
{\em Proof of Claim.} Since $B$ is a CAT(0) space, it is enough to show that the restriction of the
map $B_1\map B$ to any geodesic of $B_1$ is a local isometry. Let  $\alpha$ be any geodesic
	of $B_1$. Then the edges on $\alpha$ are mapped to geodesics 
in $B$ by Lemma \ref{iso. embedd. poly.}. Thus, the map $\alpha\map B$ is locally isometric at all points of $\alpha$
other than possibly the vertices. Now, suppose $b_1,b,b_2$ are consecutive vertices on $\alpha$.
	We know that vertices and edges of $B_1$ and $B$ correspond to cosets of vertex and edge groups in $H$ and $G$ respectively,
	and the $2$-dimensional faces of $B$ correspond to cosets of $G_{\tau}$ in $G$, where $\tau$ is the $2$-dimensional face of $Y$	(see \cite[Theorem 2.13, Chapter III.C]{bridson-haefliger}). 
 Since the map $B_1\map B$ is equivariant under the homomorphism $H\map G$
and there are only two orbits of vertices  and one orbit of edges under the $H$-action, we may assume
without loss of generality that  $b$ corresponds to the local group $G_v$ (i.e.
the stabilizer of the vertex $b$ of $\alpha$ is $G_v$), $b_1$ correspond to the $G_w$ and 
$b_2$ correspond to $g_vG_w$, where $g_v\in G_v\setminus G_e$. We note the following. (1) $g_v\not \in G_{\tau}$, (2) $b,b_1$ are in
the $2$-dimensional face corresponding to $G_{\tau}$, but (3) $b,b_2$ are in the $2$-dimensional face
corresponding to $g_vG_{\tau}$. Consequently the edges $[b,b_1]$ and $[b,b_2]$ lie in two distinct polygons
in $B$ and satisfy the hypotheses of Lemma \ref{polygon lemma 2}. 
Hence, the concatenation of these edges is a geodesic in $B$. This proves that the inclusion $B_1\map B$ is
an isometric embedding.

It then follows that (1) the homomorphism $H\map G$ is injective; (2) the $H$-action on $B_1$ is acylindrical,
since $G$-action on $B$ is acylindrical. Hence, $H$ is a hyperbolic group by \cite[Theorem 3.7]{ilyakapovichcomb}.
Also since $B_1\map B$ is an isometric embedding Mitra's criterion clearly holds, and the CT map $\partial B_1\map \partial B$
is injective by Lemma \ref{bdry basic prop}(3). Thus all the hypotheses of Theorem \ref{acyl ct thm} are verified. Hence, 
$H$ is quasiconvex in $G$.
\qed

\begin{rem}
(1) Theorem \ref{edge qc} is false for a triangle of groups in general. In this case, the natural homomorphism 
$\pi_1(\GG,\YY_1)\rightarrow \pi_1(\GG,\YY)$ need not even be injective. However, if $\YY$ is a polygon with at least $4$ edges 
and if $(\GG,\YY)$ satisfies the condition (1) of Theorem \ref{edge qc} then it follows that $(\GG,\YY)$ is developable (see \cite[12.29, p. 390, II.12]{bridson-haefliger}) and that the natural homomorphism $\pi_1(\GG,\YY_1)\rightarrow \pi_1(\GG,\YY)$ is injective. For the latter see the proof of Corollary 
\ref{graph qc}(1) below.  

Even there are examples of developable triangles of groups such that development is a CAT(0) hyperbolic space, the fundamental group $G$ of the triangle of groups is hyperbolic, but amalgamated free product corresponding to an edge is not quasiconvex in $G$, see Example \ref{not qc edge}.
	
(2) Let $\YY$ be an Euclidean polygon with at least $4$ edges. Suppose $(\GG,\YY)$ is a developable simple polygon of groups. Then $(\GG,\YY)$ satisfy the conditions (M2)-(M4) of Martin's theorem if and only if $(\GG,\YY)$ satisfies the hypotheses of Theorem \ref{edge qc}.
\end{rem}

Next we obtain a generalization of Theorem \ref{edge qc} as follows.
Suppose $(\GG,\YY)$ is a polygon of groups as in Theorem \ref{edge qc},
Suppose $\YY_2$ is a connected subgraph of the boundary of $\YY$ consisting of $n-3$ edges.
Suppose $v,w\in \YY$ are the two vertices in the complement of $\YY_2$. 
Let $J$ be the line segment inside $\YY$ joining midpoints of the edges incident on $v$ and $w$ respectively on
the opposite side of the edge $[v,w]$. Let $\YY_1$ be the edge $[v,w]$.
Let $G_1=\pi_1(\GG,\YY_1)$ and $G_2=\pi_1(\GG, \YY_2)$. Let $e_1,e_2$ be the edges incident on $v,w$ respectively
whose midpoints are joined by $J$. Let $\tau$ denote the barycenter of $\YY$ as before. 
Suppose $K=G_{e_1}\ast_{G_\tau}G_{e_2}$ be the natural amalgamated free product for the monomorphisms
$G_{\tau}\map G_{e_i}$, $i=1,2$. 
Since $(\GG,\YY_2)$ is a tree of groups, we can write $G_2$ as an amalgamated free product. 
Hence, there are the natural homomorphisms $K\map G_i$, $i=1,2$. Now we have the following corollary to
Theorem \ref{edge qc}.

\begin{cor}\label{graph qc}
(1) The homomorphisms $K\map G_i$, $i=1,2$ are injective whence we have an amalgam decomposition
$G=G_1*_K G_2$.

(2) $G_1,G_2, K$ are all quasiconvex subgroups of $G$.

(3) If $\YY'$ is any connected subgraph of $\YY_2$ then $G'=\pi_1(\GG,\YY')$ is quasiconvex in $G$.
\end{cor}

We need a little preparation for the proof of the corollary.

In \cite{GMRS}, Gitik et al. generalized the concept of malnormality and introduced the following notion of {\em height} of a subgroup of a group. 
\begin{defn}
	[Height]  The height of an infinite subgroup $H$ in $G$ is the maximal $n\in\mathbb{N}$ such that if there
	exists distinct cosets $g_1H,g_2H,...,g_nH$ such that $g_1Hg_1^{-1}\cap g_2Hg_2^{-1}\cap...\cap g_nHg_n^{-1}$ is infinite.
	The height of a finite subgroup define to be $0$.
\end{defn}
In \cite{GMRS}, the authors proved the following:
\begin{theorem}\textup{(\cite[p. 322]{GMRS})}\label{finite height of qc}
	Quasiconvex subgroups of hyperbolic groups have finite height.
\end{theorem}

The following lemma is a simple consequence of Bass-Serre theory and hence we skip its proof.
\begin{lemma}\label{finite height implies acyl}
	Suppose $(\GG,\YY)$ is a finite graph of groups with fundamental group $G$. If all the edge groups have finite height in $G$ then the action of $G$ on the Bass-Serre tree of $(\GG,\YY)$ is acylindrical.
\end{lemma}

{\em Proof of Corollary \ref{graph qc}.}
(1) follows from the results in \cite[2.15, p.25]{basscovering} since the natural homomorphisms $K\to G_i$, $i=1,2$, are injective.

(2) We know that $G$ is a hyperbolic group and all vertex groups are uniformly quasiconvex in $G$ and also
$G_1$ is quasiconvex in $G$ by Theorem \ref{edge qc}. It follows that all the vertex groups of $G$ 
have finite height by Theorem \ref{finite height of qc}. Hence, all the vertex groups of $G_2$ have finite height in $G_2$
as well. Hence, $G_2$-action on its Bass-Serre tree is acylindrical by Lemma \ref{finite height implies acyl}.
Thus by \cite[Theorem 3.7]{ilyakapovichcomb} (or by the Martin's theorem) $G_2$ is hyperbolic.

Finally, by \cite[Theorem 8.73]{kapovich-sardar}, the group $K$ is quasiconvex in $G_1$.
By Theorem \ref{edge qc} $G_1$ is quasiconvex in $G$. Hence $K$ is quasiconvex
in $G$. Thus it easily follows that $G_2$ is quasiconvex in $G$. 

(3) Once again by \cite[Theorem 8.73]{kapovich-sardar}
$G'$ is quasiconvex in $G_2$. Thus using (2) $G'$ is quasiconvex in $G$. \qed

\begin{comment}
	Since $(\GG,Y)$ is a finite graph of groups, there exists $l\in\mathbb{N}$ such that the height of each edge group is at most $l$. Let $T$ be the Bass-Serre tree of $(\GG,Y)$. Note that the $G$-stabilizer of a geodesic $\alpha$ in $T$ is the intersection of the $G$-stabilizers of edges on $\alpha$. Let if possible the action of $G$ on $T$ is not acylindrical. Then, given $k\in\mathbb{N}$ there is a geodesic in $T$ with length bigger than $k$ and the $G$-stabilizer of the geodesic is infinite. Now, choose $k$ to be sufficiently larger than $l$ such that there is a geodesic segment $\beta$ of length bigger than $k$ and  the $G$-stabilizer of $\beta$ is infinite. Since there are finitely many edge groups, the $G$-stabilizer of $\beta$ is contained in the intersection of more than $l$-conjugates of an edge group. This gives a contradiction as the height of each edge group is at most $l$.
\end{comment}
Finally we prove the following application of Corollary \ref{graph qc}.
\begin{prop}\label{virtually special}
	Suppose $(\GG,\YY)$ is a polygon of groups satisfying the hypotheses of Theorem \ref{edge qc}
such that the vertex groups are also virtually compact special. Then $G=\pi_1(\GG,\YY)$ is virtually compact special.
\end{prop}
\proof Let $\YY_1,\YY_2$, $G_1,G_2,K$ be as in Corollary \ref{graph qc}. 
Then $G=G_1*_K G_2$ is virtually compact special by \cite[Theorem 13.1]{wisehierarchy} since
$K$ is quasiconvex in $G$ by Corollary \ref{graph qc}.
\qed

\subsection{Examples}\label{examples}

We end the paper with two concrete examples which demonstrate the necessity of some of the hypotheses in 
some of the theorems of the previous sections.

 The first example shows that the natural inclusion from the universal cover of the subcomplex of groups 
to the universal cover of the complex of groups is not always a proper embedding. In particular, it shows that 
the converse of Theorem \ref{acyl ct thm} is not true in general. The second example shows that Theorem
\ref{edge qc} is false for triangles of groups. For both the examples following is the common set up:

\begin{enumerate}
\item $\YY$ is a triangle as in figure \ref{pic 4} and $(\GG,\YY)$ is a triangle of groups. 
\item We shall denote the local group for any face $\sigma$ by $G_{\sigma}$; $G_{\tau}=(1)$ for both the examples.
\item For the local homomorphisms we shall use the notation  $\phi_{ij}: G_{e_i}\map G_{v_j}$.
\end{enumerate}

\begin{figure}[h]
	\centering
	\begin{tikzpicture}
		\tikzset{vertex/.style = {shape=circle,draw,minimum size=2em}}
		\tikzset{edge/.style = {->,> = latex'}}
		\node (v) at (0,0) {$v_1$};
		\node (u) at (3,-3) {$v_2$};
		\node (w) at (-3,-3) {$v_3$};
		\node (s) at (0,-1.9) {$\tau$};
		\draw[ ultra thick] (v) to node[midway,above,sloped] {$e_3$} (u) ;
		\draw[ ultra thick] (u) to node[midway,below,sloped] {$e_1$} (w);
		\draw[ ultra thick] (v) to node[midway,above,sloped] {$e_2$} (w);		
	\end{tikzpicture}
	\caption{}
	\label{pic 4}
\end{figure}

\begin{comment}

\begin{example}
	Let $Y$ be a triangle with vertices $v_1,v_2,v_3$, edges $e_1,e_2,e_3$ and let $\tau$ be the face of $Y$, see figure \ref{pic 4}. Suppose $G_{v_1}=\langle a,b|ab=ba\rangle, G_{v_2}=\langle c,d|cd=dc\rangle, G_{v_3}=\langle e,f|ef=fe\rangle$. Suppose all the edge groups $G_{e_1},G_{e_2},G_{e_3}$ are cyclic and the images of monomorphisms from $G_{e_3}$ to $G_{v_1},G_{v_2}$ are $\langle a\rangle$,$\langle c\rangle$ respectively. Similarly, images of monomorphisms for $e_2$ are $\langle e\rangle,\langle b\rangle$ and for $e_1$ are $\langle d\rangle,\langle f\rangle$. Since the face group $G_{\tau}$ is a subgroup of the intersection of edge groups, $G_{\tau}$ is trivial. Note that $\pi_1(\GG,Y)=\langle a,b,d|ab=ba,ad=da,bd=db\rangle$. Let $Y_1=e_3$. Then, $\pi_1(\GG,Y_1)=\langle a,b,d|ab=ba,ad=da\rangle$. Clearly, the natural map $\pi_1(\GG,Y_1)\rightarrow \pi_1(\GG,Y)$ is not injective.
\end{example}
\end{comment}

\begin{example}\label{not proper embedding}
\begin{enumerate}
\item Let $$G_{v_1}=\langle a,b|a^2,b^2\rangle, G_{v_2}=\langle c,d|c^2,d^2\rangle, G_{v_3} =\langle x,y,z|x^2,y^2,z^2\rangle.$$ 
\item Let $$G_{e_i}=\langle t_i| t_i^2\rangle, i=1,2,3.$$ 
The local homomorphisms are defined as follows.
	
\item  
\[
\begin{array}{cc}
\phi_{31}:G_{e_3}\map G_{v_1}, & \phi_{32}:G_{e_3}\map G_{v_2}\\
\hspace{0.5cm} t_3\mapsto a & \hspace{0.5cm} t_3\mapsto c
\end{array}
\]

\item 
\[
\begin{array}{cc}
\phi_{12}:G_{e_1}\map G_{v_2}, & \phi_{13}:G_{e_1}\map G_{v_3}\\
\hspace{0.5cm} t_1\mapsto d & \hspace{0.5cm} t_1\mapsto y
\end{array}
\]

\item 
\[
\begin{array}{cc}
\phi_{21}:G_{e_2}\map G_{v_1} & \phi_{23}:G_{e_2}\map G_{v_3}\\
\hspace{0.5cm} t_2\map b & \hspace{0.5cm} t_2\mapsto x
\end{array}
\]

\end{enumerate}  
Note that $(\GG,\YY)$ is developable, $G=\pi_1(\GG,\YY)\simeq \langle a,b,d,z|a^2,b^2,d^2,z^2\rangle$ is hyperbolic and that
$G_{v_i}$'s are all quasiconvex in $\pi_1(\GG,\YY)$. Thus, the universal cover $B$ of $(\GG,\YY)$ is 
a hyperbolic space by Proposition \ref{mj-dahmani}. Let $\YY_1=e_3$ and let $(\GG,\YY_1)$ be the restriction of $(\GG,\YY)$ to $\YY_1$. 
Then $\pi_1(\GG,\YY_1)=G_{v_1} *_{G_{e_3}} G_{v_2}$ is quasiconvex in $\pi_1(\GG,\YY)$. The universal cover $B_1$ of 
$(\GG,\YY_1)$ is the Bass-Serre tree of $\pi_1(\GG,\YY_1)$. Consider the two vertices $v=G_{v_1}$ and $w_n=dbdb...db$ $(n-times)G_{v_2}$ 
of $B_1$. Note that $d_{B_1}(v,w_n)=n+1$. On the other hand, by construction of $B$ \textup{\cite[Theorem 2.13,III.C]{bridson-haefliger}}, one sees that $d_B(v,w_n)=2$. Hence, the natural map $B_1\rightarrow B$ is not a proper embedding and clearly there is no Cannon-Thurston map 
for the inclusion $B_1\map B$.
\end{example}

The next example shows that the conclusion of Theorem \ref{edge qc} is false for a triangle of groups. For the definition of hyperbolic automorphism of a free group, one is referred to \cite{BF}.

\begin{example}\label{not qc edge}
\begin{enumerate}
\item Suppose $\mathbb F=\langle x,y,z\rangle$ is a free group on $3$ generators and $\phi$ is a hyperbolic automorphism 
(see \cite{BF}) of $\mathbb F$.
Let $G_{v_1}$ be the semidirect product $\mathbb Z \ltimes \mathbb F$ where we use letter $t$ to denote a generator of $\mathbb Z$, so that 
$$G_{v_1}=\langle x,y,z,t| twt^{-1}=\phi(w), w=x,y,z\rangle.$$ We note that $G_{v_1}$ is a hyperbolic group by the main result of \cite{BF}.
	
\item Suppose $$G_{v_3}=\langle a,b,c\rangle, G_{v_2}=\langle d,e\rangle$$ are free groups of rank $3$ and $2$ respectively.

\item Suppose $$G_{e_i}=\langle t_i \rangle, i=1,  3$$ are infinite cyclic. 
\item Suppose $$G_{e_2}=\langle u_1,u_2 \rangle $$ is a free group on two generators.
\item Suppose the local homomorphisms are defined as follows:
$$\phi_{21}: G_{e_2}\map G_{v_1}, u_1\mapsto a,u_2\mapsto b;$$
$$ \phi_{23}: G_{e_2}\map G_{v_3}, u_1\mapsto z, u_2\mapsto x;$$

\[
\begin{array}{cc}
\phi_{12}:G_{e_1}\map G_{v_2}, &  \phi_{13}: G_{e_1}\map G_{v_3}\\
\hspace{0.5cm} t_1\mapsto c & \hspace{0.5cm} t_1\mapsto d
\end{array} 
\]

\[
\begin{array}{cc}
\phi_{31}:G_{e_3}\map G_{v_1}, &  \phi_{32}:G_{e_3}\map G_{v_2}\\
\hspace{0.5cm} t_3\mapsto e &  \hspace{0.5cm} t_3\mapsto y
\end{array}
\]
\end{enumerate}    
We note that $\pi_1(\GG,\YY)$ can be naturally identified with $G_{v_1}\ast \mathbb Z$. 
Let $\YY_1=e_1$ and let $(\GG,\YY_1)$ be the restriction of $(\GG,\YY)$. 
Then $\pi_1(\GG,\YY_1)=\mathbb F\ast \mathbb Z$ where $\mathbb F$ is the normal free subgroup of $G_{v_1}$ of rank $3$.
Clearly, $\pi_1(\GG,\YY_1)$ is not a quasiconvex subgroup of $\pi_1(\GG,\YY)$.
\end{example}

\begin{comment}
	
Next, we give an example of a developable triangle of groups such that the universal cover is not a CAT(0) space.
\begin{example}
	Let $Y$ be a Euclidean equilateral triangle as in figure \ref{pic 4}. Define a triangle of groups $(\GG,Y)$ in the following way:
	
	Assume that $G_{v_1}=\langle a,b,g|a^2,b^2,g^2,ab=ba,ag=ga,bg=gb\rangle$, $G_{v_2}=\langle c,d|c^2,d^2\\,cd=dc\rangle$, $G_{v_3} =\langle e,f,h|e^2,f^2,h^2,ef=fe,eh=he,fh=hf\rangle$. Assume that all the edge groups are of order $2$. Let the images of monomorphisms from $G_{e_3}$ to $G_{v_1},G_{v_2}$ be $\langle a \rangle,\langle c \rangle$. Similarly, images of monomorphisms for edge $e_1$ are $\langle d\rangle,\langle f\rangle$ and for $e_2$ are $\langle e\rangle,\langle b\rangle$. The face group $G_{\tau}$ is trivial. Note that Gersten-Stallings's angle \cite{gerstein-stalling} at each vertex is $\dfrac{\pi}{2}$. Thus, this triangle of groups is not non-positively curved in the sense of Gersten-Stallings \cite{gerstein-stalling}. Now, the universal cover $B$ of $(\GG,Y)$ is a $M_{\kappa}^2$ complex for $\kappa= 0$. By our choice of local groups, we have an injective loop of length $4$ in the link of a vertex of $B$ . Thus, the development does not satisfy the link condition by \cite[Lemma 5.6, II.5]{bridson-haefliger}. Hence, the development is not a CAT(0) space.
\end{example}
\end{comment}

\bibliography{acyl-cplx}
\bibliographystyle{alpha}

\end{document}